\documentclass[11pt, oneside]{amsart}   	
\usepackage{geometry}                		
\usepackage[parfill]{parskip}    		
\usepackage{graphicx}
\usepackage{amssymb,amsmath}
\newtheorem{theorem}{Theorem}
\newtheorem{lemma}[theorem]{Lemma}
\newtheorem{corollary}[theorem]{Corollary}

\theoremstyle{definition}
\newtheorem{remark}[theorem]{Remark}
\newtheorem{exm}[theorem]{Example}
\newtheorem{problem}[theorem]{Problem}
\newtheorem{definition}[theorem]{Definition}
\usepackage{epstopdf}
\numberwithin{equation}{section}
\numberwithin{theorem}{section}

\author[P. Dragnev]{P. D. Dragnev $^{\dagger}$}
\address{Department of Mathematical Sciences
Indiana-Purdue University
Fort Wayne, IN 46805, USA }
\email{dragnevp@ipfw.edu}
\thanks{\noindent $^{\dagger}$ The research of this author was supported, in part, by a Simons Foundation grant no. 282207.}

\author[D.~Hardin]{D.~P.~Hardin$^*$}
\address{Center for Constructive Approximation, Department of Mathematics, 
Vanderbilt University,
Nashville, TN 37240, USA  }
\email{doug.hardin@vanderbilt.edu}

\author[E.~Saff]{E.~B.~Saff$^*$}
\email{edward.b.saff@vanderbilt.edu}
\thanks{\noindent $^*$ The research of these authors was supported, in part,
by the U.~S.~National Science Foundation under grants  DMS-1109266 and DMS-1412428.
}
\author[N.~Zorii]{N.~Zorii$^{**}$}
\address{Department of Mathematical Physics, Institute of Mathematics,
National Academy of Sciences of Ukraine, Tereshchenkivska 3, 01601,
Kyiv-4, Ukraine}
\email{natalia.zorii@gmail.com}
\thanks{
\noindent $^{**}$ The research of this author was supported, in part, by the Scholar-in-Residence program at IPFW
}
\thanks{The authors express their gratitude to Erwin Schr\"{o}dinger International Institute for providing  conducive research atmosphere during their stay when part of this manuscript was written.}

\begin{document}
\title{Minimum Riesz energy problems for a condenser with ``touching plates"}


\date{\today}							

\begin{abstract}

We consider minimum energy problems in the presence of an external field for
a condenser with ``touching plates" $A_1$ and $A_2$ in $\mathbb{R}^{n}$, $n\geqslant3$, relative to the $\alpha$-Riesz kernel $|x-y|^{\alpha-n}$, $0<\alpha\leqslant2$. An
intimate relationship between such problems and minimal $\alpha$-Green energy problems for positive measures on~$A_1$ is shown.
We obtain sufficient and/or necessary conditions for the solvability of these problems in both the unconstrained and
the constrained settings, investigate the properties of minimizers, and prove their uniqueness.
Furthermore, characterization theorems in terms of variational inequalities for the weighted potentials are established.
The approach applied is mainly based on the establishment of a perfectness-type property for the $\alpha$-Green kernel with $0<\alpha<2$
which enables us, in particular, to analyze the existence of the $\alpha$-Green equilibrium measure of a set. The results obtained are
illustrated by several examples.
\end{abstract}
\maketitle

\section{Introduction}

The purpose of the paper is to study minimum energy problems in the presence of an external field for
a condenser $\mathbf A$ with touching oppositely-charged plates $A_1$ and $A_2$ in $\mathbb{R}^{n}$, $n\geqslant3$, relative to the $\alpha$-Riesz kernel $|x-y|^{\alpha-n}$,
$0<\alpha\leqslant2$.
The difficulties appearing in the course of our investigation are caused by the fact that a short-circuit between $A_1$ and~$A_2$ might occur, for the Euclidean distance
between these conductors is zero.

Therefore, it is meaningful to ask what kind of conditions on the objects in question could prevent such a phenomenon so that
a minimizer for the corresponding $\alpha$-Riesz energy problem might
exist.
One of the ideas, to be discussed for this purpose, is to impose  upper constraints on the charges of the touching conductors.

We establish sufficient and/or necessary conditions for the existence of minimizing measures for both the unconstrained and
the constrained problems, and prove their uniqueness. The conditions obtained are expressed in geometric-potential terms for~$A_1$ and~$A_2$, or in measure theory terms for
the constraints under consideration, or in terms of variational inequalities for the weight\-ed potentials. We also provide a detailed analysis of the supports of the
minimizers.

The approach developed in the paper is based on a newly discovered important relationship between, on the one hand, minimum $\alpha$-Riesz energy problems over signed measures associated with a
condenser~$\mathbf A$ and, on the other hand,
minimum energy problems for nonnegative measures on~$A_1$ relative to the $\alpha$-Green function~$g^\alpha_{D}$ of the domain $D:=\mathbb R^n\setminus A_2$.

Regarding the latter problems, crucial to the arguments applied in their investigation is the pre-Hil\-bert structure on the linear space~$\mathcal E_{g^\alpha_{D}}(D)$ of all (signed)
Borel measures on~$D$ with finite $g^\alpha_{D}$-Green energy and, most importantly, a completeness theorem for certain metric subspaces of~$\mathcal
E_{g^\alpha_{D}}(D)$ with nonnegative elements. This completeness theorem enables us, in particular, to analyze the existence of the $\alpha$-Green equilibrium measure of a
set.

To formulate precisely the problems in question, we first need to introduce several notions, to discuss relations between them, and to recall some well-known results; this is the purpose of the next section. The scheme of the rest of the paper is described at the end of Section~\ref{sec-pr}, after the formulations of the problems  (see
Problems~\ref{pr} and~\ref{pr2}).

\section{Basic notions; relations between them. Known results}

\subsection{Measures, energies, potentials.} Let $\mathrm X$ be a locally compact Hausdorff space, to be
specified below, and $\mathfrak M=\mathfrak M(\mathrm X)$ the linear
space of all Radon measures~$\mu$ on~$\mathrm X$, equipped with the
{\it vague\/} ($=${\it weak\/}$^*$) topology, i.e.~the topology of
pointwise convergence on the class of all
continuous functions on~$\mathrm X$ with compact
support. We denote by~$\mu^+$ and~$\mu^-$ the positive and the negative parts in the Hahn--Jordan decomposition of a measure $\mu\in\mathfrak M(\mathrm X)$, respectively, and by
$S^\mu_{\mathrm X}$ its support. Given~$\mu$ and a $\mu$-measurable function~$\psi$, for the sake of brevity we shall write
$\langle\psi,\mu\rangle:=\int\psi\,d\mu$.\footnote{When introducing notation, we assume
the corresponding object on the right to be well-defined.}

A {\it
kernel\/} $\kappa(x,y)$ on $\mathrm X$ is a symmetric, lower
semicontinuous function $\kappa:\mathrm X\times\mathrm
X\to[0,\infty]$. Given $\mu,\mu_1\in\mathfrak M$, let
$E_\kappa(\mu,\mu_1)$ and $U_\kappa^\mu(\cdot)$ denote the {\it mutual
energy\/} and the {\it potential\/} relative to the kernel~$\kappa$,
respectively, i.e.
\begin{eqnarray*}
E_\kappa(\mu,\mu_1)\!\!\!\!&:=&\!\!\!\!\int\kappa(x,y)\,d(\mu\otimes\mu_1)(x,y),\\
U_\kappa^\mu(x)\!\!\!\!&:=&\!\!\!\!\int\kappa(x,y)\,d\mu(y).
\end{eqnarray*}
For $\mu=\mu_1$, the mutual energy $E_\kappa(\mu,\mu_1)$ defines the
{\it energy\/} $E_\kappa(\mu,\mu)=:E_\kappa(\mu)$.

\subsection{Strictly positive definite kernels. Capacities.}
Throughout this section, a kernel~$\kappa$ is assumed to be {\it strictly positive definite\/}, which means that $E_\kappa(\mu)$, $\mu\in\mathfrak M(\mathrm X)$, is
nonnegative whenever defined and $E_\kappa(\mu)=0$ implies $\mu=0$.
Then the collection $\mathcal E_\kappa=\mathcal E_\kappa(\mathrm X)$
of all $\mu\in\mathfrak M$ with
$E_\kappa(\mu)<\infty$ forms a pre-Hil\-bert space with the
scalar product $E_\kappa(\mu,\mu_1)$ and the norm
$\|\mu\|_\kappa:=\sqrt{E_\kappa(\mu)}$ (see~\cite{F1}). The topology
on~$\mathcal E_\kappa$ defined by~$\|\cdot\|_\kappa$ is called {\it
strong\/}. The following lemma from the geometry of the
pre-Hilbert space~$\mathcal E_\kappa$ is often useful (see~\cite[Lemma~4.1.1]{F1}).

\begin{lemma}\label{convex}Let\/
$\Gamma$ be a convex subset of\/~$\mathcal E_\kappa$. If there
exists\/ $\mu_0\in\Gamma$ with minimal norm\/
\[\|\mu_0\|_\kappa=\inf_{\mu\in\Gamma}\,\|\mu\|_\kappa,\] then such a minimal
element is unique. Moreover,
\[\|\mu-\mu_0\|_\kappa^2\leqslant\|\mu\|_\kappa^2-\|\mu_0\|_\kappa^2\quad\text{for all \ }\mu\in\Gamma.\]
\end{lemma}

Given a set $B\subset\mathrm X$, let $\mathfrak M^+(B)$ be the
convex cone of all nonnegative measures concentrated in~$B$, and let
$\mathfrak M^+(B,b)$, $b>0$, consist of all $\mu\in\mathfrak M^+(B)$
with $\mu(B)=b$. We also write $\mathfrak M^+:=\mathfrak M^+(\mathrm
X)$, $\mathcal E^+_\kappa(B):=\mathcal E_\kappa\cap\mathfrak
M^+(B)$, $\mathcal E^+_\kappa:=\mathcal E^+_\kappa(\mathrm X)$, and
$\mathcal E^+_\kappa(B,b):=\mathcal E_\kappa\cap\mathfrak M^+(B,b)$.

Let~$C_\kappa(B)$ denote the {\it interior capacity\/} of~$B$
relative to a kernel~$\kappa$, defined by\footnote{As usual, the
infimum over the empty set is taken to be~$+\infty$. We put
$1\bigl/(+\infty)=0$ and $1\bigl/0=+\infty$.}
\begin{equation}\label{cap}1\bigl/C_\kappa(B):=w_\kappa(B):=\inf_{\mu\in\mathcal
E_\kappa^+(B,1)}\,E_\kappa(\mu).\end{equation}  Note that, in
consequence of Lemma~\ref{convex}, a measure $\lambda_B=\lambda^\kappa_B\in\mathcal
E_\kappa^+(B,1)$ with minimal energy $\|\lambda_B\|^2_\kappa=w_\kappa(B)$ is unique (provided it exists).

Following Fuglede~\cite{F1}, we call a
kernel~$\kappa$ {\it perfect\/} if any strong Cauchy sequence in~$\mathcal E_\kappa^+$ converges strongly and, in addition, the strong topology on~$\mathcal E_\kappa^+$ is finer than the induced vague topology on~$\mathcal E_\kappa^+$.
Note that then the metric space~$\mathcal E_\kappa^+$ is complete in the induced strong topology. What is also important is that the solution~$\lambda^\kappa_B$ to the minimum
energy problem appeared in~(\ref{cap}) exists, provided that $\kappa$ is perfect, $B$ is closed, and $0<C_\kappa(B)<\infty$ (see~\cite[Theorem~4.1]{F1}).

 If $f:\mathrm X\to[-\infty,\infty]$ is an {\it external field\/}, then the $f$-{\it weighted potential\/} $W^\mu_{\kappa,f}$ and the $f$-{\it weight\-ed energy\/} $G_{\kappa,f}(\mu)$ of
 $\mu\in\mathcal E_\kappa(\mathrm X)$ are respectively given by
\begin{align*}W^\mu_{\kappa,f}(x)&:=U_\kappa^\mu(x)+f(x),\\
G_{\kappa,f}(\mu)&:=E_\kappa(\mu)+2\langle f,\mu\rangle=\langle W^\mu_{\kappa,f}+f,\mu\rangle.\end{align*}
We also define \[\mathcal E_{\kappa,f}(\mathrm X):=\bigl\{\mu\in\mathcal E_\kappa(\mathrm X): \ G_{\kappa,f}(\mu)<\infty\bigr\}.\]

\subsection{$\alpha$-Riesz and $\alpha$-Green kernels. Balayage.}\label{sec:bala}
Fix $n\geqslant3$, a domain~$D\subset\mathbb R^n$, and $\alpha\in(0,2]$. In the rest of the paper, unless stated otherwise, one of the following two cases is assumed to hold:
either $\mathrm X=\mathbb R^n$ and $\kappa$ is the $\alpha$-Riesz
kernel $\kappa_\alpha(x,y):=|x-y|^{\alpha-n}$ (where $|x-y|$ is the Euclidean distance in~$\mathbb R^n$ between $x$ and~$y$), or $\mathrm X=D$ and
$\kappa$ is the generalized $\alpha$-Green function~$g^\alpha_D$ of~$D$, defined by\footnote{The function $g^\alpha_D(x,y)$, $x,y\in D$, is nonnegative and
symmetric (see~\cite{Fr} and~\cite[Chapter~IV, Section~5]{L}). It is also lower semicontinuous, which follows from the lower
semicontinuity of the $\alpha$-Riesz kernel and the continuity
of~$U_{\kappa_\alpha}^{\beta^\alpha_{D^c}\varepsilon_y}(x)$, $x,y\in
D$, the latter in turn being a consequence of
$U_{\kappa_\alpha}^{\beta^\alpha_{D^c}\varepsilon_y}(x)=
U_{\kappa_\alpha}^{\beta^\alpha_{D^c}\varepsilon_x}(y)$, $x,y\in
D$. Hence, $g^\alpha_D$ can be treated as a kernel on the locally compact space~$D$.}
\begin{equation}\label{G}g^\alpha_D(x,y)=U_{\kappa_\alpha}^{\varepsilon_y}(x)-
U_{\kappa_\alpha}^{\beta^\alpha_{D^c}\varepsilon_y}(x)\quad
\text{for all \ }x,y\in D,\end{equation} where $\varepsilon_y$ denotes the unit Dirac measure at a
point~$y$ and $\beta^\alpha_{D^c}$ the $\alpha$-Riesz
balayage onto $D^c:=\mathbb
R^n\setminus D$ (cf.~\cite[Chapter~IV, Section~5]{L}, or see just below).

      To avoid triviality, assume each component of~$D^c$ to have nonzero $\alpha$-Riesz capacity. Note that, if $\alpha=2$ and $D$ is regular in the sense of the solvability of
      the (classical) Dirichlet problem, then $g^2_D$ is, in fact, the classical Green function of~$D$.

Let $Q$ be a given closed subset of~$\mathbb R^n$. By definition, the $\alpha$-Riesz {\it balayage\/} measure
$\beta^\alpha_Q\mu$ of $\mu\in\mathfrak M(\mathbb R^n)$ onto~$Q$ is
supported by~$Q$ and satisfies the relation
\begin{equation*}\label{def}U_{\kappa_\alpha}^{\beta^\alpha_Q\mu}(x)=U_{\kappa_\alpha}^{\mu}(x)\quad\text{n.e.~in
\ }Q,\end{equation*} where ``{\it n.e.}" ({\it nearly everywhere\/}) means that the equality holds everywhere in~$Q$ except for a subset with $\alpha$-Riesz capacity zero. Such
a~$\beta^\alpha_Q\mu$ exists and it
is unique among the $C$-{\it ab\-sol\-utely continuous\/} measures $\nu\in\mathfrak M(\mathbb R^n)$, namely those that $\nu(K)=0$ for every compact $K\subset\mathbb R^n$ with
$C_{\kappa_\alpha}(K)=0$. Throughout the paper, when speaking of the
$\alpha$-Riesz balayage measure, we always mean exactly this one.
Then, by~\cite[Chapter~IV, Section~5]{L},
\begin{equation}\label{repr}\beta^\alpha_Q\mu=\int\beta^\alpha_Q\varepsilon_y\,d\mu(y).\end{equation}
If $\mu\geqslant0$, then also
$U_{\kappa_\alpha}^{\beta^\alpha_Q\mu}(x)\leqslant
U_{\kappa_\alpha}^{\mu}(x)$ for all $x\in\mathbb R^n$. If, moreover,
$\mu\in\mathcal E^+_{\kappa_\alpha}(\mathbb R^n)$, then \[E_{\kappa_\alpha}(\beta^\alpha_Q\mu)\leqslant E_{\kappa_\alpha}(\mu)\] and also
\begin{equation}\label{bala}\bigl\|\mu-\beta^\alpha_Q\mu\bigr\|_{\kappa_\alpha}<\bigl\|\mu-\nu\bigr\|_{\kappa_\alpha}\quad\text{for all \ }
\nu\in\mathcal E^+_{\kappa_\alpha}(Q), \
\nu\not=\beta^\alpha_Q\mu,\end{equation}
so that $\beta^\alpha_Q\mu$ is, in fact, the orthogonal projection of $\mu$ in the pre-Hilbert space~$\mathcal E_{\kappa_\alpha}(\mathbb R^n)$ onto the convex cone $\mathcal
E^+_{\kappa_\alpha}(Q)$.

It is well known (see \cite{L}) that, if $\mu\in\mathfrak M^+(\mathbb R^n)$ is {\it bounded\/} (i.e., $\mu(\mathbb R^n)<\infty$), then
\[\beta^\alpha_Q\mu(\mathbb R^n)\leqslant\mu(\mathbb R^n).\]
This general fact is specified by the following assertion (see \cite[Theorem~4]{Z1}; for $\alpha=2$, see also \cite[Theorem~B]{Z0}).

\begin{theorem}\label{th-bala} $Q$ is not\/ $\alpha$-thin at the Alexandroff point\/~$\omega$ of\/~$\mathbb R^n$ if and only if, for every bounded measure\/ $\mu\in\mathfrak
M^+(\mathbb R^n)$,
\begin{equation}\label{leq-bala}\beta^\alpha_Q\mu(\mathbb R^n)=\mu(\mathbb R^n).\end{equation}
\end{theorem}

By definition, $Q$ is {\it not $\alpha$-thin
at\/}~$\omega$ if $Q^*$, the inverse of~$Q$ relative to the unit sphere $S(0,1):=\{x\in\mathbb R^n: |x|=1\}$, is not $\alpha$-thin at $x=0$, or equivalently
(see \cite[Theorem~5.10]{L}), if $x=0$ is an
$\alpha$-regular point for~$Q^*$.\footnote{For $\alpha=2$, this definition is due to Brelot
(see~\cite{Brelo2}; cf.~also~\cite{F3,Hayman2}). For $\alpha\in(0,2)$, such a notion has been introduced in~\cite{Z1}.}

If $\mu\in\mathfrak M(D)$ and $U_{g^\alpha_D}^\mu(x)$ is well
defined at $x\in D$, then, due to (\ref{G}) and~(\ref{repr}),
\begin{equation}\label{UG}U_{g^\alpha_D}^\mu(x)=U_{\kappa_\alpha}^{\mu}(x)-
U_{\kappa_\alpha}^{\beta^\alpha_{D^c}\mu}(x).\end{equation} Likewise,
\begin{equation}
\label{EG}E_{g^\alpha_D}(\mu)=E_{\kappa_\alpha}\bigl(\mu-\beta^\alpha_{D^c}\mu\bigr)=
E_{\kappa_\alpha}(\mu)-E_{\kappa_\alpha}\bigl(\beta^\alpha_{D^c}\mu\bigr),\end{equation}
whenever $E_{g^\alpha_D}(\mu)$ is well defined.
See~\cite[Chapter~XI, Section~10]{Br} and~\cite[Chapter~IV, Section~1, n$^\circ$\,2]{L}, where (\ref{EG}) has been shown for
$\alpha=2$. For $\alpha<2$, the proof is based on~(\ref{G})
and~(\ref{repr}) and runs in a way similar to that in~\cite{L}.

It is well known that the $\alpha$-Riesz kernel is strictly positive
definite and, moreover, perfect (see, e.g., \cite{Car,D1,D2,F1,L});
hence, the metric space $\mathcal E^+_{\kappa_\alpha}(\mathbb R^n)$
is complete in the induced strong topology. However, by
Cartan~\cite{Car}, the whole pre-Hilbert space~$\mathcal
E_{\kappa_\alpha}(\mathbb R^n)$ is, in general, strongly incomplete, and this is the case even for the Coulomb kernel
$\kappa_2(x,y)=|x-y|^{-1}$ on~$\mathbb R^3$; compare with
Theorem~\ref{th-complete} below.

Thus, in consequence of~(\ref{EG}), the $\alpha$-Green kernel is strictly positive
definite as well. In the case $\alpha=2$, the $2$-Green kernel~$g^2_D$ is
actually known to be perfect (see~\cite{E1,F1,L}), so that, by~\cite[Theorem~4.1]{F1}, the measure $\lambda_B\in\mathcal E^+_{g^2_D}(B,1)$ with minimal $g^2_D$-Green energy
\[\|\lambda_B\|^2_{g^2_D}=w_{g^2_D}(B)=\bigl[C_{g^2_D}(B)\bigr]^{-1}\] exists provided that $B$ is closed and $0<C_{g^2_D}(B)<\infty$. For similar
results related to the $\alpha$-Green kernel with $\alpha<2$, see
Theorems~\ref{alphaeq1} and~\ref{alphaeq} below.

From now on we shall often write simply $\alpha$ instead
of~$\kappa_\alpha$ if it serves as an index. E.g.,
$C_\alpha(\cdot)=C_{\kappa_\alpha}(\cdot)$ and~$C_{g^\alpha_D}(\cdot)$ denote the
$\alpha$-Riesz and the $\alpha$-Green interior capacities of a set,
respectively.

\begin{lemma}\label{lemma:equiv}For any\/ $B\subset D$, $C_\alpha(B)=0$ if
and only if\/ $C_{g_D^\alpha}(B)=0$.\end{lemma}

\begin{proof} We need the following general facts related to an arbitrary strictly positive definite kernel~$\kappa$ on a locally compact space~$\mathrm X$. First of all, for
any $B\subset\mathrm X$,
\begin{equation}\label{ex}C_\kappa(B)=\sup_{K\in\{K\}_B}\,C_\kappa(K),\end{equation} where $\{K\}_B$ consists of all compact subsets of~$B$ (see~\cite{F1}). Further, for a
compact set
$K\subset\mathrm X$ one has $C_\kappa(K)<\infty$ and hence,
by~\cite[Theorem~2.5]{F1},
\begin{equation}\label{w}C_\kappa(K)=\sup\,\mu(K),\end{equation}
where $\mu$ ranges over all nonnegative measures
supported by~$K$ with the additional property
\[U_\kappa^\mu(x)\leqslant1\quad\text{for all \ }x\in S^\mu_{\mathrm X}.\]

Now we apply representation
(\ref{w}) to a compact set $K\subset D$ and each of the $\alpha$-Riesz and the $\alpha$-Green kernels, which is possible in view of their strict positive definiteness.
Since for every $\mu\in\mathfrak M^+(K)$,
$U_{\kappa_\alpha}^\mu(x)-U_{g_D^\alpha}^\mu(x)$ is bounded on~$K$
in consequence of~(\ref{UG}), the lemma for $B=K$ follows. To prove the lemma for any $B\subset D$, it is thus left to apply~(\ref{ex}).\end{proof}

Let $B\subset D$. By Lemma~\ref{lemma:equiv}, if some expression $E(x)$ is valid n.e.~in~$B$, then $C_{g^\alpha_D}(N)=0$, $N$~being the set of all $x\in B$ with $E(x)$ not to
hold; and also the other way around.

\subsection{Condensers. Existence of minimizers.} By a {\it condenser\/} in $\mathbb R^n$ we mean an ordered pair
$\mathbf B=(B_1,B_2)$ of nonintersecting sets $B_1,B_2\subset\mathbb
R^n$ (so far of arbitrary topological structure), treated as the
{\it positive\/} and the {\it negative\/} plates of~$\mathbf B$, respectively. Define
\[\mathfrak M(\mathbf B):=\bigl\{\nu\in\mathfrak M(\mathbb R^n): \ \nu^+\in\mathfrak M^+(B_1), \
\nu^-\in\mathfrak M^+(B_2)\bigr\},\]
\[\mathcal E_\alpha(\mathbf B):=\mathfrak M(\mathbf B)\cap\mathcal E_\alpha(\mathbb R^n).\]
Then the following theorem on
the strong completeness is true (see~\cite[Theorem~1]{ZUmzh}; compare with~\cite{Car} or \cite[Theorem~1.19]{L}).\footnote{In fact, Theorem~\ref{th-complete}
 holds true also for $\alpha\in(2,n)$; see~\cite[Theorem~1]{ZUmzh}. Its proof is based on Deny's
theorem~\cite{D1} stating that, for the Riesz kernels, $\mathcal
E_\alpha$~can be completed by making use of distributions with finite
energy. Regarding the history of the question, see \cite[Theorem~A]{Z0} and \cite[Theorem~1]{Z1}.}

\begin{theorem}\label{th-complete}
Assume\/ $B_1$ and\/ $B_2$ to be closed in\/ $\mathbb R^n$. Then the metric space\/~$\mathcal
E_\alpha(\mathbf B)$ is complete in the induced strong topology, and
the strong convergence in this space implies the vague convergence
to the same limit.
\end{theorem}

From now on we fix a (particular) condenser $\mathbf A=(A_1,A_2)$ in $\mathbb R^n$ with the plates
$A_2:=D^c$ and $A_1:=F$, where
$F\subseteqq D$ is closed in the relative topology of~$D$ and
$C_\alpha(F)>0$. Thus (see also Section~\ref{sec:bala}),
\begin{equation}\label{nonzero2}C_\alpha(A_i)>0\quad\text{for all \ }i=1,\,2.\end{equation}

Also fix a unit two-dimensional numerical vector ${\mathbf 1}=(1,1)$, and define
\[\mathcal E_\alpha(\mathbf A,\mathbf 1):=\bigl\{\mu\in\mathcal E_\alpha(\mathbf
A): \ \mu^+(A_1)=\mu^-(A_2)=1\bigr\}.\] In view of~(\ref{nonzero2}), the class $\mathcal E_\alpha(\mathbf A,\mathbf 1)$ is nonempty and, hence, it makes sense to
consider the variational problem on the existence of
$\lambda_{\mathbf A}\in\mathcal E_\alpha(\mathbf A,\mathbf 1)$ with
\begin{equation}\label{inf}E_\alpha(\lambda_{\mathbf A})=
\inf_{\mu\in\mathcal E_\alpha(\mathbf A,\mathbf
1)}\,E_\alpha(\mu)\quad\bigl({}=:E_\alpha(\mathbf A,\mathbf
1)\bigr).\end{equation} In particular, the following theorem on the solvability
holds (see~\cite[Theorems~5,~7]{Z1}).\footnote{The first result of this type was obtained in~\cite[Theorem~1]{Z0}, where $\alpha=2$ and $A_1$ was assumed to be compact in~$D$.
See also~\cite[Theorem 12]{ZUmzh} where
Theorem~\ref{z1} has been generalized to any $\alpha\in(0,n)$ and any $\mathbf a=(a_1,a_2)$
with $a_i>0$, $i=1,2$. Instead of the balayage technique, which implicitly appears in Theorem~\ref{z1}, for $\alpha\in(2,n)$ one should use the operator of orthogonal projection
in the pre-Hilbert space~$\mathcal E_\alpha$ onto the convex cone~$\mathcal E^+_\alpha(D^c)$.}

\begin{theorem}\label{z1} Let
\begin{equation*}\label{dist}
\inf_{x\in F, \ y\in D^c}\,|x-y|>0.\end{equation*} If, moreover, $C_\alpha(F)<\infty$ then,
for a solution to the minimum energy problem\/~{\rm(\ref{inf})} to exist, it is necessary and sufficient that either\/ $C_\alpha(D^c)<\infty$
or\/ $D^c$ be not\/ $\alpha$-thin at\/~$\omega$.\footnote{We refer
to~\cite{Z0,ZPot2} for an example of a set with infinite
Newtonian capacity, though $2$-thin at~$\omega$.}
\end{theorem}

In the paper, we are mainly interested in the case
\begin{equation*}\label{int}\overline{F}\cap\partial_{\overline{\mathbb R^n}}D
\ne\varnothing,\end{equation*}
where $\overline{\mathbb R^n}:=\mathbb R^n\cup\{\omega\}$ is the one-point compactification of $\mathbb R^n$ and $\overline{F}:={\mathrm C\ell}_{\overline{\mathbb R^n}}F$.
Then, in general, the infimum
value in~(\ref{inf}) can not be achieved among $\mu\in\mathcal E_\alpha(\mathbf A,\mathbf
1)$. Using the physical interpretation, which is possible for the Coulomb
kernel, a short-circuit between the plates
of the condenser might occur.

Therefore, it is meaningful to investigate what kind of conditions on the objects under consideration would prevent such a phenomenon, and
a minimizer in the corresponding minimum $\alpha$-Riesz energy problem for the condenser~$\mathbf A$ would, nevertheless,
exist. One of the ideas, to be discussed, is to find out such an upper constraint on the measures (charges) from $\mathcal
E_\alpha^+(F,1)$ which would not allow the ``blow-up" effect between $F$ and $D^c$. Note that we do not intend to impose any constraint on the measures on~$D^c$.

Assume also the measures from $\mathcal
E_\alpha^+(F)$ to be influenced by some external field~$f$, while $\langle f,\mu\rangle=0$ for all $\mu\in\mathcal
E_\alpha^+(D^c)$. Then, what kind of external fields, acting on the charges on~$F$ only, would still guarantee the existence of minimizers?

Recently a similar problem for the logarithmic kernel in $\mathbb R^2$ has been investigated by Beckermann and Gryson~\cite[Theorem~2.2]{BC}. Our study is related to the Riesz
kernels, and the results obtained and the approaches applied are rather different from those in~\cite{BC}.

\section{Constrained and unconstrained minimum energy problems}\label{sec-pr}
When speaking of the external field~$f$, we shall tacitly assume that at least one of the following Cases~I or~II takes place:
\begin{itemize}
\item[\rm I.] {\it $f\bigl|_F$ is lower semicontinuous, and it is\/ ${}\geqslant0$ unless\/ $F$ is compact, while\/ $f(x)=0$ n.e.~in\/}~$D^c$;
\item[\rm II.] {\it $f(x)=U_\alpha^{\zeta-\beta^\alpha_{D^c}\zeta}(x)$ for all\/ $x\in\mathbb R^n$, where a\/ {\rm(}signed\/{\rm)} measure\/ $\zeta\in\mathcal E_{g^\alpha_D}(D)$ is given.}
\end{itemize}
Note that Case~II can certainly be reduced to Case~I provided that $\zeta^-=0$.

The values $G_{\alpha,f}(\mu)$, $\mu\in\mathcal E_\alpha(\mathbf A)$, and $G_{g^\alpha_D,f}(\nu)$, $\nu\in\mathcal E^+_{g^\alpha_D}(F)$, are well defined
and
\begin{align}\label{alphafinite}
G_{\alpha,f}(\mu)&=\|\mu\|^2_\alpha+2\langle f,\mu^+\rangle>-\infty,\\
\label{wen}G_{g^\alpha_D,f}(\nu)&=\|\nu\|^2_{g^\alpha_D}+2\langle f,\nu\rangle>-\infty.
\end{align}
If Case~II holds, then these values can alternatively be expressed in the form
\begin{align}\label{alphafinite0}
G_{\alpha,f}(\mu)&=\|\mu\|^2_\alpha+2E_\alpha(\zeta-\beta^\alpha_{D^c}\zeta,\mu)=\|\mu+\zeta-\beta^\alpha_{D^c}\zeta\|^2_\alpha-\|\zeta-\beta^\alpha_{D^c}\zeta\|^2_\alpha,\\
\label{wen0}G_{g^\alpha_D,f}(\nu)&=\|\nu\|^2_{g^\alpha_D}+2E_{g^\alpha_D}(\zeta,\nu)=
\|\nu+\zeta\|^2_{g^\alpha_D}-\|\zeta\|^2_{g^\alpha_D}.
\end{align}

Write \[\mathcal E_{\alpha,f}(\mathbf A,\mathbf 1):=\mathcal E_\alpha(\mathbf A,\mathbf 1)\cap\mathcal E_{\alpha,f}(\mathbb R^n)\quad\text{and}\quad\mathcal
E^+_{g^\alpha_D,f}(F,1):=\mathcal E^+_{g^\alpha_D}(F,1)\cap\mathcal E_{g^\alpha_D,f}(D);\]
these classes of measures are convex. It is seen from~(\ref{alphafinite0}) and~(\ref{wen0}) that, in Case~II,
\begin{equation}\label{identtt}\mathcal E_{\alpha,f}(\mathbf A,\mathbf 1)=\mathcal E_\alpha(\mathbf A,\mathbf 1)\text{ \ and \ }\mathcal E^+_{g^\alpha_D,f}(F,1)=\mathcal
E^+_{g^\alpha_D}(F,1).\end{equation}

We denote by $\mathfrak C(F)$ the collection of all $\xi\in\mathfrak M^+(D)$ with the properties
\[S^\xi_D=F\quad\text{and}\quad\xi(F)>1;\]
such~$\xi$ will be treated as {\it constraints}\/ for measures from the class~$\mathfrak M^+(F,1)$. Let $\mathfrak C_0(F)$ consist of all bounded $\xi\in\mathfrak C(F)$, i.e.,
with $\xi(F)<\infty$. Given $\xi\in\mathfrak C(F)$,  write
\begin{align*}
\mathcal E^{\xi}_{\alpha,f}(\mathbf A,\mathbf 1)&:=
\bigl\{\mu\in\mathcal E_{\alpha,f}(\mathbf A,\mathbf 1): \
\mu^+\leqslant\xi\bigr\},\\
\mathcal E^{\xi}_{g^\alpha_D,f}(F,1)&:=
\bigl\{\nu\in\mathcal E_{g^\alpha_D,f}^+(F,1): \
\nu\leqslant\xi\bigr\},\end{align*}
where $\nu_1\leqslant\nu_2$
means that $\nu_2-\nu_1$ is a nonnegative Borel measure.

To combine (where this is possible) assertions related to extremal problems in both the constrained and unconstrained settings, we accept the notations
\[\mathcal E^{\infty}_{\alpha,f}(\mathbf A,\mathbf 1):=\mathcal E_{\alpha,f}(\mathbf A,\mathbf 1)\quad\text{and}\quad
\mathcal E^{\infty}_{g^\alpha_D,f}(F,1):=\mathcal E^+_{g^\alpha_D,f}(F,1).\]

In all that follows, we consider a fixed $\sigma\in\mathfrak C(F)\cup\{\infty\}$, where the formal formula $\sigma=\infty$ means that {\it no}\/ upper constraint is
allowed,\footnote{It is natural to set $\nu\leqslant\infty$ for all $\nu\in\mathfrak M^+(F)$.} and we define
\begin{align*}G_{\alpha,f}^{\sigma}(\mathbf A,\mathbf
1)&:=\inf_{\mu\in\mathcal E^{\sigma}_{\alpha,f}(\mathbf A,\mathbf
1)}\,G_{\alpha,f}(\mu),\\
G_{g^\alpha_D,f}^\sigma(F,1)&:=\inf_{\nu\in\mathcal
E_{g^\alpha_D,f}^\sigma(F,1)}\,G_{g^\alpha_D,f}(\nu).\end{align*}
In the case $\sigma=\infty$, the upper index~$\infty$ will often be omitted, i.e., we shall write
\[G_{\alpha,f}(\mathbf A,\mathbf
1):=G_{\alpha,f}^{\infty}(\mathbf A,\mathbf
1)\quad\text{and}\quad G_{g^\alpha_D,f}(F,1):=G_{g^\alpha_D,f}^\infty(F,1).\]

Note that each of $G_{\alpha,f}^{\sigma}(\mathbf A,\mathbf
1)$ and $G_{g^\alpha_D,f}^\sigma(F,1)$ is ${}>-\infty$. One can also see from~(\ref{nonzero2}) and Lemma~\ref{lemma:equiv} that any of the classes $\mathcal
E^{\sigma}_{\alpha,f}(\mathbf A,\mathbf
1)$ and $\mathcal
E_{g^\alpha_D,f}^\sigma(F,1)$ is nonempty if and only if so is the other, and therefore the following two assumptions are equivalent:\footnote{See Lemmas~\ref{autom0},~\ref{autom} and Remark~\ref{rem-aut} below, providing necessary and/or sufficient conditions for (\ref{riesz}) and~(\ref{green}) to hold.}
\begin{align}\label{riesz}
G_{\alpha,f}^{\sigma}(\mathbf A,\mathbf
1)&<\infty,\\
\label{green}
 G_{g^\alpha_D,f}^\sigma(F,1)&<\infty.
\end{align}

\begin{problem}\label{pr} Under condition\/ {\rm(\ref{riesz})}, does there  exist\/ $\lambda^{\sigma}_{\mathbf A}\in\mathcal
E^{\sigma}_{\alpha,f}(\mathbf A,\mathbf 1)$ whose\/ $f$-weighted\/ $\alpha$-Riesz energy is minimal in this class, i.e.
\begin{equation}\label{inf2}G_{\alpha,f}(\lambda^{\sigma}_{\mathbf A})=
G_{\alpha,f}^{\sigma}(\mathbf A,\mathbf
1)\,?\end{equation}\end{problem}

Problem~\ref{pr} turns out to be intricately related to the following one.

\begin{problem}\label{pr2} Under condition\/ {\rm(\ref{green})}, does there  exist\/ $\lambda^\sigma_{F}\in\mathcal
E_{g^\alpha_D,f}^\sigma(F,1)$ whose\/ $f$-weighted\/ $\alpha$-Green energy is minimal in this class, i.e.
\begin{equation}\label{CF}
G_{g^\alpha_D,f}(\lambda^\sigma_{F})=G_{g^\alpha_D,f}^\sigma(F,1)\,?\end{equation}
\end{problem}

Note that, if $\sigma=\infty$ and $f=0$, then Problem~\ref{pr} is in fact reduced to the minimum energy problem~(\ref{inf}), while Problem~\ref{pr2} to that appeared
in~(\ref{cap}) for $B=F$ and $\kappa=g^\alpha_D$.

The rest of the paper is organized as follows. Sufficient and/or necessary conditions for Problems~\ref{pr} and~\ref{pr2} to be solvable are established in Sections~\ref{geom},
\ref{sec-main-var-green} and~\ref{sec-main-var-riesz}. In Section~\ref{geom} they are formulated either in measure theory terms for the constraints under consideration, or in
geometric-potential terms for~$F$ and~$D^c$, while in Sections~\ref{sec-main-var-green} and~\ref{sec-main-var-riesz} they are given in terms of variational inequalities for the
$f$-weight\-ed potentials. Sections~\ref{sec-main-var-green} and~\ref{sec-main-var-riesz} provide also a detailed analysis of the supports of the minimizers. The results
obtained are proved in Sections~\ref{silv-bound-proof}, \ref{infcap-proof}, \ref{lemma-main-proof} and~\ref{ss-1}, and they are illustrated by the examples in
Section~\ref{sec-ex}.

E.g., by Theorem~\ref{lemma-main}, in both Cases~I and~II, the condition $C_{g_D^\alpha}(F)<\infty$ is close to be sufficient for Problem~\ref{pr2} to be solvable for every
$\sigma\in\mathfrak C(F)\cup\{\infty\}$,  while according to Theorem~\ref{infcap}, it is also necessary for this to happen provided Case~II with $\zeta\geqslant0$ holds. However, if we
restricted our analysis to the constraints from the class~$\mathfrak C_0(F)$ then, in~Case~I, Problem~\ref{pr2} would already be always solvable (see Theorem~\ref{silv-bound}).
Further, if we assume $D^c$ to be not $\alpha$-thin at~$\omega$, then all this remains true for Problem~\ref{pr} as well. See Section~\ref{geom} for the strict formulations of
the results just described.

A crucial key in the proofs is Theorem~\ref{alphaeq1}, providing us with a perfectness-type result for the $\alpha$-Green kernel, where $\alpha<2$ (cf.~\cite{E1} for
$\alpha=2$).
It makes it possible, in particular, to establish Theorem~\ref{alphaeq} on the
existence of the $\alpha$-Green equilibrium measure of a set.

The uniqueness of solutions to Problems~\ref{pr} and~\ref{pr2} is shown by Lemma~\ref{uniqueness} in the next section. Another assertion of this section,
Lemma~\ref{lemma-aux''}, discovers an intimate relationship between their solvability (or unsolvability), as well as their minimizers (provided they exist), which turns out to
be a powerful tool in the proofs of the above-mentioned results.

\section{Auxiliary assertions}\label{sec-aux}

\begin{lemma}\label{uniqueness} A solution to Problem\/~{\rm\ref{pr}}, as well as that to Problem\/~{\rm\ref{pr2}}, is unique\/ {\rm(}provided it exists\/{\rm)}.\end{lemma}

\begin{proof} We shall verify the latter part of the lemma. Assume there exist two solutions to Problem~\ref{pr2}, $\lambda^\sigma_{F}$ and $\hat{\lambda}^\sigma_{F}$.
Since the class
$\mathcal E^\sigma_{g^\alpha_D,f}(F,1)$ is convex, from (\ref{wen}) we get
\[4G^\sigma_{g^\alpha_D,f}(F,1)\leqslant4G_{g^\alpha_D,f}\Bigl(\frac{\lambda^\sigma_{F}+\hat{\lambda}^\sigma_{F}}{2}\Bigr)=
\|\lambda^\sigma_{F}+\hat{\lambda}^\sigma_{F}\|^2_{g^\alpha_D}+
4\langle f,\lambda^\sigma_{F}+\hat{\lambda}^\sigma_{F}\rangle.\]
Applying the parallelogram
identity in $\mathcal E_{g^\alpha_D}(D)$, we obtain
\begin{equation*}0\leqslant\|\lambda^\sigma_{F}-\hat{\lambda}^\sigma_{F}\|^2_{g^\alpha_D}\leqslant-
4G^\sigma_{g^\alpha_D,f}(F,1)+
2G_{g^\alpha_D,f}(\lambda^\sigma_{F})+2G_{g^\alpha_D,f}(\hat{\lambda}^\sigma_{F}),\end{equation*}
 so that $\|\lambda^\sigma_{F}-\hat{\lambda}^\sigma_{F}\|_{g^\alpha_D}=0$ by (\ref{CF}).  As $g^\alpha_D$ is strictly positive definite,
$\lambda^\sigma_{F}=\hat{\lambda}^\sigma_{F}$.

Likewise, the former part of the lemma can be proved with the help of the convexity of the class $\mathcal E^\sigma_{\alpha,f}(\mathbf A,\mathbf 1)$ and the pre-Hilbert
structure in the space~$\mathcal E_\alpha(\mathbb R^n)$.\end{proof}

\begin{lemma}\label{lemma-aux''}Assume\/ $D^c$ to be not\/ $\alpha$-thin at\/~$\omega$.  Then for every\/ $\sigma\in\mathfrak C(F)\cup\{\infty\}$,
\begin{equation}\label{equality}G^\sigma_{g^\alpha_D,f}(F,1)=G^{\sigma}_{\alpha,f}(\mathbf A,\mathbf 1).\end{equation}
In addition, the solution to Problem\/~{\rm\ref{pr}} exists if and only if so does that to Problem\/~{\rm\ref{pr2}}, and then they are related to each other by the formula
\begin{equation}\label{reprrr}
\lambda^{\sigma}_{\mathbf A}=\lambda^\sigma_{F}-\beta^\alpha_{D^c}\lambda^\sigma_{F}.
\end{equation}
\end{lemma}

\begin{proof} We begin by establishing the inequality
\begin{equation}\label{ineq}G^\sigma_{g^\alpha_D,f}(F,1)\geqslant G^{\sigma}_{\alpha,f}(\mathbf A,\mathbf 1).\end{equation}
Having assumed $G^\sigma_{g^\alpha_D,f}(F,1)<\infty$, we fix $\nu\in\mathcal E^\sigma_{g^\alpha_D,f}(F,1)$; then, by~(\ref{leq-bala}) for $Q=D^c$,
\[\nu-\beta^\alpha_{D^c}\nu\in\mathcal E^\sigma_{\alpha,f}(\mathbf A,\mathbf 1).\]
On account of (\ref{EG}), (\ref{alphafinite}) and (\ref{wen}), we therefore get
\[G_{g^\alpha_D,f}(\nu)=\|\nu\|^2_{g^\alpha_D}+2\langle f,\nu\rangle=\|\nu-\beta^\alpha_{D^c}\nu\|^2_\alpha+2\langle
f,\nu\rangle=G_{\alpha,f}(\nu-\beta^\alpha_{D^c}\nu)\geqslant G_{\alpha,f}^\sigma(\mathbf A,\mathbf 1).\]
Since $\nu\in\mathcal E^\sigma_{g^\alpha_D,f}(F,1)$ has been chosen arbitrarily, this yields (\ref{ineq}).

On the other hand,  for any $\mu\in\mathcal E^\sigma_{\alpha,f}(\mathbf A,\mathbf 1)$ we have $\mu^+\in\mathcal
E_{g^\alpha_D,f}^\sigma(F,1)$. Thus, due to~(\ref{bala}),
\begin{align}\label{hope}G_{\alpha,f}(\mu)&=
\|\mu\|^2_\alpha+2\langle f,\mu^+\rangle\geqslant\|\mu^+-\beta^\alpha_{D^c}\mu^+\|^2_\alpha+2\langle f,\mu^+\rangle\\
&{}=\|\mu^+\|^2_{g^\alpha_{D^c}}+2\langle f,\mu^+\rangle=G_{g^\alpha_D,f}(\mu^+)\geqslant G^\sigma_{g^\alpha_D,f}(F,1).\notag\end{align}
In view of the arbitrary choice of $\mu\in\mathcal
E^\sigma_{\alpha,f}(\mathbf A,\mathbf 1)$, this proves~(\ref{equality}) when combined with~(\ref{ineq}).

Let now $\lambda^\sigma_{F}\in\mathcal
E_{g^\alpha_D,f}^\sigma(F,1)$ satisfy~(\ref{CF}). Then, in consequence of (\ref{leq-bala}), \[\widehat\mu:=\lambda^\sigma_{F}-\beta^\alpha_{D^c}\lambda^\sigma_{F}\in\mathcal
E^\sigma_{\alpha,f}(\mathbf A,\mathbf 1).\]
Substituting $\widehat\mu$ instead of $\mu$ in relation~(\ref{hope}), we see that all the inequalities therein are, in fact, equalities. Therefore,
by~(\ref{equality}),
\[G_{\alpha,f}(\widehat\mu)=G^\sigma_{g^\alpha_{D^c},f}(F,1)=
G^\sigma_{\alpha,f}(\mathbf A,\mathbf 1),\]
so that the measure $\lambda^\sigma_{\mathbf A}$, defined by~(\ref{reprrr}), solves Problem~\ref{pr}.

To complete the proof, assume further that $\lambda^{\sigma}_{\mathbf A}=\lambda^+-\lambda^-\in\mathcal
E^{\sigma}_{\alpha,f}(\mathbf A,\mathbf 1)$ satisfies~(\ref{inf2}). Then, by~(\ref{equality}) and~(\ref{hope}), the latter with $\lambda^{\sigma}_{\mathbf A}$ instead of~$\mu$,
\begin{align*}G^\sigma_{g^\alpha_D,f}(F,1)&=G_{\alpha,f}(\lambda^{\sigma}_{\mathbf A})\geqslant\|\lambda^+-\beta^\alpha_{D^c}\lambda^+\|_\alpha^2+2\langle f,\lambda^+\rangle\\
&{}=\|\lambda^+\|^2_{g^\alpha_{D^c}}+2\langle f,\lambda^+\rangle=G_{g^\alpha_D,f}(\lambda^+)\geqslant G^\sigma_{g^\alpha_D,f}(F,1).
\end{align*}
Hence, all the inequalities here are, in fact, equalities. This shows that $\lambda^\sigma_{F}:=\lambda^+$ solves Problem~\ref{pr2} and, on account of~(\ref{bala}), also that
$\lambda^-=\beta^\alpha_{D^c}\lambda^+=\beta^\alpha_{D^c}\lambda^\sigma_{F}$.\end{proof}

\begin{lemma}\label{lequiv} Assume\/ {\rm(\ref{green})} holds. For a measure\/
$\lambda=\lambda^\sigma_{F}\in\mathcal
E_{g^\alpha_D,f}^\sigma(F,1)$ to solve Problem\/~{\rm\ref{pr2}}, it is necessary and sufficient that
\begin{equation}\label{lchar}\bigl\langle
W_{g^\alpha_D,f}^\lambda,\nu-\lambda\bigr\rangle\geqslant0\quad\mbox{for all \ }
\nu\in\mathcal E_{g^\alpha_D,f}^\sigma(F,1).\end{equation}\end{lemma}

\begin{proof}By direct calculation, for any $\nu,\mu\in\mathcal E_{g^\alpha_D,f}^\sigma(F,1)$ and any
$h\in(0,1]$ we obtain
\begin{equation}\label{mainin}G_{g^\alpha_D,f}\bigl(h\nu+(1-h)\mu\bigr)-G_{g^\alpha_D,f}\bigl(\mu\bigr)=2h\bigl\langle
W_{g^\alpha_D,f}^\mu,\nu-\mu\bigr\rangle+h^2\bigl\|\nu-\mu\bigr\|^2_{g^\alpha_D}.\end{equation}  If $\mu=\lambda^\sigma_{F}$ solves Problem~\ref{pr2}, then the left (hence, the
right) side
of~(\ref{mainin}) is~${}\geqslant0$, for the class $\mathcal E_{g^\alpha_D,f}^\sigma(F,1)$ is convex, which leads to~(\ref{lchar}) by
letting $h\to0$.

Conversely, if (\ref{lchar}) holds, then
(\ref{mainin}) with $\mu=\lambda$ and $h=1$ gives $G_{g^\alpha_D,f}(\nu)\geqslant
G_{g^\alpha_D,f}(\lambda)$ for all $\nu\in\mathcal E_{g^\alpha_D,f}^\sigma(F,1)$, which means that $\lambda=\lambda^\sigma_{F}$ solves Problem~\ref{pr2}.\end{proof}

\begin{lemma}\label{lemma-cont}$G_{g^\alpha_D,f}(\cdot)$ is vaguely lower semicontinuous on\/ $\mathcal E_{g^\alpha_D}^+(D)$ if Case\/~{\rm I} takes place, and otherwise, it is
strongly continuous.\end{lemma}

\begin{proof}If Case I holds, then the lemma follows from~\cite{F1} (see Section~1.1 and Lemma~2.2.1 therein), while otherwise it is a direct consequence of relation~(\ref{wen0}).\end{proof}

Lemmas~\ref{autom0} and~\ref{autom} below provide sufficient and/or necessary conditions that guarantee~(\ref{riesz}) and~(\ref{green}) (compare with Lemmas~{\rm4} and~{\rm5} from~\cite{Z5}). From now on we write  \begin{equation}\label{f0}F_0:=\bigl\{x\in F: \ f(x)<\infty\bigr\}.\end{equation}

\begin{lemma}\label{autom0}Let\/ $\sigma=\infty$. Then\/~{\rm(\ref{green})}
{\rm(}hence, also {\rm(\ref{riesz})}{\rm)} holds if and only if\/ $C_\alpha(F_0)>0$.\end{lemma}

\begin{proof} Suppose first that $C_\alpha(F_0)>0$. On account of \cite[Lemma~2.3.3]{F1}, then one can choose a compact set $K\subset F$ with $C_\alpha(K)>0$ so that $f(x)\leqslant M<\infty$ for all $x\in K$. In turn, this yields that there exists $\nu\in\mathcal E^+_{g^\alpha_D}(K,1)$ with $G_{g^\alpha_D,f}(\nu)<\infty$, and (\ref{green}) follows.

To prove the necessary part of the lemma, assume, on the contrary, that $C_\alpha(F_0)=0$. Then for every $\mu\in\mathcal E^+_{g^\alpha_D}(F,1)$, we necessarily have $G_{g^\alpha_D,f}(\nu)=\infty$, which contradicts~(\ref{green}).\end{proof}

\begin{definition}\label{admissible} $\xi\in\mathfrak C(F)$ is called {\it admissible\/} if its restriction to any compact subset of~$F$ has finite $\alpha$-Riesz (hence,
$\alpha$-Green) energy. Let $\mathcal A(F)$ consist of all admissible constraints.\end{definition}

When considering admissibility of a measure, the parameter $\alpha$ and the set $F$  should be clear in each context. Observe that, for a constraint $\xi\in\mathfrak C(F)$ to be admissible, it is sufficient that its $\alpha$-Riesz potential be continuous.
Also note that any $\xi\in\mathcal A(F)$ is $C$-absolutely continuous.

\begin{lemma}\label{autom} If\/ $\xi\in\mathcal A(F)$, then\/~{\rm(\ref{green})}
{\rm(}hence, also\/~{\rm(\ref{riesz})}{\rm)} holds provided that\/ $\xi(F_0)>1$.
\end{lemma}
\begin{proof}Choose a compact set $K\subset F$ so that $\xi(K)>1$ and $f(x)\leqslant M<\infty$ for all $x\in K$. Then ${\xi|_K}\bigl/{\xi(K)}\in\mathcal E^\xi_{g^\alpha_D,f}(F,1)$,
which yields~(\ref{green}).\end{proof}

\begin{remark}\label{rem-aut}If Case~II takes place, then $f(x)$ is finite n.e.~in~$F$ and, hence,
Lemma~\ref{autom0} (similarly, Lemma~\ref{autom}) remains true with $C_\alpha(F_0)>0$ (respectively, $\xi(F_0)>1$) dropped.\end{remark}

\section{Criteria of the solvability, given either in measure theory terms for~$\sigma$, or in geometric-potential terms for~$F$ and~$D^c$}\label{geom}
Throughout this section and Sections~\ref{sec-main-var-green} and~\ref{sec-main-var-riesz}, assume~(\ref{riesz}) or, equivalently,~(\ref{green}) to be satisfied. See Lemmas~\ref{autom0},~\ref{autom} and Remark~\ref{rem-aut} above, providing necessary and/or sufficient conditions for these to hold.

\begin{theorem}\label{silv-bound}If Case\/~{\rm I} takes place, then Problem\/~{\rm\ref{pr2}} is\/ {\rm(}uniquely\/{\rm)} solvable for every constraint\/ $\xi\in\mathfrak
C_0(F)$.\end{theorem}

In the next theorem, the following condition on the geometry of the condenser~$\mathbf A$ is required to hold:
\begin{itemize}
\item[$(\ast)$] \ {\it If\/ $\alpha<2$, then\/ $\overline{F}\cap\partial_{\overline{\mathbb R^n}}D$ consists of at most one point\/.}
\end{itemize}
Note that we do not impose any restriction on $\overline{F}\cap\partial_{\overline{\mathbb R^n}}D$ provided that $\alpha=2$.

\begin{theorem}\label{lemma-main} If, moreover,
\begin{equation}\label{finite}C_{g_D^\alpha}(F)<\infty\end{equation}
then, in both Cases\/~{\rm I} and\/~{\rm II}, Problem\/~{\rm\ref{pr2}} is\/ {\rm(}uniquely\/{\rm)}
solvable for every\/ $\sigma\in\mathfrak C(F)\cup\{\infty\}$.\footnote{Compare
with~\cite[Theorem~2.2]{ZCAOT} and~\cite[Theorem~8.1]{ZPot2}.}
\end{theorem}

\begin{theorem}\label{infcap} Suppose that Case\/~{\rm II} with\/ $\zeta\geqslant0$ takes place. If, moreover,
$C_{g^\alpha_D}(F)=\infty$, then
Problem\/~{\rm\ref{pr2}} is unsolvable for every\/ $\sigma\in\mathfrak C(F)\cup\{\infty\}$ such that\/ $\sigma\geqslant\xi_0$, where\/ $\xi_0\in\mathfrak C(F)\setminus\mathfrak
C_0(F)$ is properly chosen.
\end{theorem}

Combining Theorems~\ref{lemma-main} and~\ref{infcap} shows that, if assumption~$(\ast)$ and Case~II with $\zeta\geqslant0$ both hold, then (\ref{finite}) is necessary and
sufficient for Problem~\ref{pr2} to be solvable for every $\sigma\in\mathfrak C(F)\cup\{\infty\}$.

\begin{theorem}\label{th-main'} Assume\/
$D^c$ to be not\/ $\alpha$-thin at\/~$\omega$. Under the hypotheses of Theorem\/~{\rm\ref{silv-bound}} {\rm(}similarly, Theorems\/~{\rm\ref{lemma-main}}
or\/~{\rm\ref{infcap}}{\rm)}, its conclusion remains true for Problem\/~{\rm\ref{pr}} as well.
\end{theorem}

Indeed, Theorem \ref{th-main'} is obtained from Theorems~\ref{silv-bound}--\ref{infcap} with the help of Lemma~\ref{lemma-aux''}.

In the next two sections we shall examine properties of the $f$-weighted potentials and the supports of the minimizers~$\lambda^\sigma_F$ and~$\lambda^\sigma_{\mathbf A}$, whose
existence has been ensured, e.g., by Theorems~\ref{silv-bound}, \ref{lemma-main}, and~\ref{th-main'}.

\section{Variational inequalities for the $f$-weighted $\alpha$-Green potentials}\label{sec-main-var-green}

This section provides necessary and/or
sufficient conditions for the solvability of Problem~\ref{pr2} in terms of variational inequalities for the $f$-weight\-ed $\alpha$-Green potentials. It also presents a detailed
analysis of properties of the supports of the minimizers.

Following \cite[p.~164]{L}, we denote by $\breve{F}$ the {\it reduced kernel\/} of~$F$, i.e.
\begin{equation}\label{red}\breve{F}:=\bigl\{x\in F: \ C_\alpha\bigl(B(x,\varepsilon)\cap
F\bigr)>0\text{ \ for every  $\varepsilon>0$}\bigr\}.\end{equation} Here $B(x,\varepsilon):=\{y\in\mathbb R^n: \ |y-x|<\varepsilon\}$. Observe that, if the constraint under consideration is admissible, then necessarily
$F=\breve{F}$.

To simplify the formulations of the results obtained, throughout this section and Section~\ref{sec-main-var-riesz} we assume $\partial D$ to be simultaneously the boundary of
the (open) set~${\rm Int}\,D^c$. Here, the boundary and the interior are considered relative to~$\mathbb R^n$. Notice that then $m_n(D^c)>0$, where $m_n$ is the $n$-dimensional
Lebesgue measure.

\subsection{Variational inequalities in the constrained $\alpha$-Green minimum energy problems}\label{sec-var1} We start by studying Problem~\ref{pr2} in the constrained case
(i.e., for $\sigma\ne\infty$).
In this section, we consider $\xi\in\mathcal A(F)$ and assume that $\xi(F_0)>1$. Note that, for any $\nu\in\mathcal
E^+_{g^\alpha_D}(F)$,  $W^\nu_{g^\alpha_D,f}(x)$ is well defined and~${}\ne-\infty$ n.e.~in~$F$, while it is finite n.e.~in~$F_0$ (see~(\ref{f0})).

\begin{theorem}\label{th-main''} Let Case\/~{\rm I} take place. Then a measure\/ $\lambda\in\mathcal E_{g^\alpha_D,f}^\xi(F,1)$ solves
Problem\/~{\rm\ref{pr2}} if and only if
there exists\/ $w_\lambda\in\mathbb R$ possessing the following two properties:
\begin{align}W^\lambda_{g^\alpha_D,f}(x)&\geqslant w_\lambda\quad(\xi-\lambda)\mbox{-a.e.~in\ }F,\label{b1}\\
W^\lambda_{g^\alpha_D,f}(x)&\leqslant w_\lambda\quad\mbox{for all \ }x\in S^{\lambda}_D.\label{b2}\end{align}
\end{theorem}

\begin{corollary}\label{cor-02} Let\/ $f\bigl|_D=U^\chi_{g^\alpha_D}$ for some\/ $\chi\in\mathfrak M^+(D)$. If\/ $\lambda$ solves Problem\/~{\rm\ref{pr2}}, then
\begin{equation*}\label{clll}C_\alpha\bigl(\partial D\cap S^{\xi-\lambda}_{\mathbb R^n}\bigr)=0.\end{equation*}
\end{corollary}

When speaking of the {\it non-weighted case\/} $f=0$, we simply write
\[\mathcal E_{g^\alpha_D}^\xi(F,1):=\bigl\{\nu\in\mathcal E_{g^\alpha_D}^+(F,1): \ \nu\leqslant\xi\bigr\}.\] Then Problem~\ref{pr2} is in fact reduced to that on the existence of $\lambda_0\in\mathcal E_{g^\alpha_D}^\xi(F,1)$ with
\begin{equation}\label{non-w}E_{g^\alpha_D}(\lambda_0)=\inf_{\nu\in\mathcal E_{g^\alpha_D}^\xi(F,1)}\,E_{g^\alpha_D}(\nu).\end{equation}

\begin{corollary}\label{cor-0} Let\/ $f=0$. A measure\/ $\lambda_0\in\mathcal E_{g^\alpha_D}^\xi(F,1)$ solves Problem\/~{\rm\ref{pr2}} if and only if
there exists\/ $w'_{\lambda_0}\in(0,\infty)$ such
that
\begin{align}U^{\lambda_0}_{g^\alpha_D}(x)&=w'_{\lambda_0}\quad(\xi-\lambda_0)\mbox{-a.e.~in\ }F,\label{b1'}\\
U^{\lambda_0}_{g^\alpha_D}(x)&\leqslant w'_{\lambda_0}\quad\mbox{for all \ }x\in D.\label{b2'}\end{align}
If, moreover, $\alpha<2$, then also
\begin{equation}\label{identity}S^{\lambda_0}_D=F.\end{equation}
\end{corollary}

On account of the uniqueness of a solution to Problem~\ref{pr2}, such a~$w'_{\lambda_0}$ is unique (provided it exists). Integrating~(\ref{b1'}) with respect to~$\xi-\lambda_0$, we
get
\begin{equation}\label{same}w'_{\lambda_0}=\frac{E_{g^\alpha_D}(\lambda_0,\xi-\lambda_0)}{(\xi-\lambda_0)(D)}.\end{equation}

\subsection{Variational inequalities in the unconstrained $\alpha$-Green minimum energy problems}\label{sec-var2} Throughout this section, it is assumed that $\sigma=\infty$.
We proceed with criteria of the solvability of Problem~\ref{pr2}, given in terms of variational inequalities for the $f$-weight\-ed $\alpha$-Green potentials. In the
unconstrained case, the  results obtained take a simpler form if compare with those in the constrained case, while provide us with much more detailed information about the
potentials and the supports of the minimizers.

\begin{theorem}\label{th-infty}Suppose that Case\/~{\rm I} takes place. For\/ $\lambda\in\mathcal E_{g^\alpha_D,f}^+(F,1)$ to solve Problem\/~{\rm\ref{pr2}}, it is necessary and
sufficient that there exist\/ $w_f\in\mathbb R$ possessing the properties
\begin{align}W^\lambda_{g^\alpha_D,f}(x)&\geqslant w_f\quad\mbox{n.e.~in\ }F,\label{F1}\\
W^\lambda_{g^\alpha_D,f}(x)&\leqslant w_f\quad\mbox{for all \ }x\in S^{\lambda}_D.\label{F2}\end{align}
Such a number\/ $w_f$ is unique {\rm(}provided it exists\/{\rm)} and can be given by the formula
\begin{equation*}w_f=\bigl\langle W^\lambda_{g^\alpha_D,f},\lambda\bigr\rangle.\label{winfty}\end{equation*}
\end{theorem}

\begin{corollary}\label{cor-infty}Let Problem\/~{\rm\ref{pr2}} be solvable. Then the following two assertions hold:
\begin{itemize}
\item[\rm(a)] If Case\/~{\rm II} with\/ $\zeta\geqslant0$ takes place, then\/ $C_{g^\alpha_D}(F)<\infty$;
\item[\rm(b)] If\/ $f\bigl|_D=U^\chi_{g^\alpha_D}$ for\/ $\chi\in\mathfrak M^+(D)$, then\/ $C_\alpha\bigl(\partial D\cap{\mathrm C\ell}_{\mathbb R^n}\breve{F}\bigr)=0$.
\end{itemize}
\end{corollary}

\begin{corollary}\label{cor-0'-infty}Let\/ $f=0$. Then\/ $\lambda_F\in\mathcal E_{g^\alpha_D}^+(F,1)$ solves Problem\/~{\rm\ref{pr2}} if and only if there exists a number\/
$w\in(0,\infty)$ admitting the properties
\begin{align}U^{\lambda_F}_{g^\alpha_D}(x)&=w\quad\mbox{n.e.~in\ }F,\label{w0-infty}\\
U^{\lambda_F}_{g^\alpha_D}(x)&\leqslant w\quad\mbox{for all \ }x\in D.\label{w1-infty}\end{align}
Such a number\/ $w$ is unique {\rm(}provided it exists\/{\rm)} and can be written in the form
\begin{equation*}\label{wF1}w=E_{g^\alpha_D}(\lambda_F)=w_{g^\alpha_D}(F)=\bigl[C_{g^\alpha_D}(F)\bigr]^{-1}.\end{equation*}
      Furthermore, if the minimizer\/~$\lambda_F$ exists, then it is the unique measure in the class\/ $\mathcal E_{g^\alpha_D}^+(F,1)$ whose\/ $\alpha$-Green potential is
      constant n.e.~in\/~$F$. Namely, if\/ $\nu\in\mathcal
E^+_{g_D^\alpha}(F,1)$ and\/ $U_{g_D^\alpha}^\nu(x)=c$ n.e.~in\/~$F$, where\/ $c\in\mathbb R$,
then\/ $\nu=\lambda_F$.\end{corollary}

Recall that the reduced kernel $\breve{F}$ of $F$ has been defined by~(\ref{red}).  For the sake of simplicity, in the following assertion we assume that, if $\alpha=2$, then $D\setminus F$ is connected.

\begin{corollary}\label{cor-sup}Let\/ $f=0$. If\/ $\lambda_F$ solves Problem\/~{\rm\ref{pr2}}, then, in addition to\/~{\rm(\ref{w0-infty})} and\/~{\rm(\ref{w1-infty})}, we have
\begin{equation}\label{less}U_{g^\alpha_D}^{\lambda_F}(x)<w\quad\text{for all \
}x\in D\setminus\breve{F}.\end{equation}
Furthermore,
\begin{equation}\label{desc-sup}
 S^{\lambda_F}_{D}=\left\{
\begin{array}{cll} \breve{F} & \mbox{if} & \alpha<2,\\ \partial_D\breve{F} & \mbox{if} & \alpha=2.\\ \end{array} \right.
\end{equation}
\end{corollary}

\subsection{Duality relation between
non-weighted constrained and weighted unconstrained $\alpha$-Green minimum energy problems}\label{sec-var7} Throughout this section, $F$ is compact. Consider the non-weighted Problem~\ref{pr2} with a constraint $\xi\in\mathfrak C(F)$ whose potential $U^\xi_{g^\alpha_D}(x)$ is continuous.\footnote{When speaking of a continuous function, we
understand that the values are {\it finite\/} numbers.} Note that then $E_{g^\alpha_D}(\xi)<\infty$. Also assume that $\lambda_0$ is its solution, i.e.~both
$\lambda_0\in\mathcal E^\xi_{g^\alpha_D}(F,1)$ and (\ref{non-w}) hold. Write
\[\theta:=q(\xi-\lambda_0),\quad\text{where \ }q:=\frac{1}{\xi(F)-1}. \]

Combining Corollary \ref{cor-0} and Theorem \ref{th-infty} allows us to formulate the following result.
\begin{theorem}\label{th-dual'} The measure\/ $\theta$ solves Problem\/~{\rm\ref{pr2}} with the external field\/ $f(x):=-qU^\xi_{g^\alpha_D}(x)$ in
 both the unconstrained and the\/ $q\xi$-constrained settings, i.e.
\[\theta\in\mathcal E^{q\xi}_{g^\alpha_D,f}(F,1)\subset\mathcal E^+_{g^\alpha_D,f}(F,1)\text{ \ and \ }G_{g^\alpha_D,f}(\theta)=G_{g^\alpha_D,f}^{q\xi}(F,1)=G_{g^\alpha_D,f}(F,1).\] Moreover,
\begin{align}\label{D1} W^{\theta}_{g^\alpha_D,f}(x)&=-qw'_{\lambda_0} \quad \mbox{on\ }S^{\theta}_D,\\
W^{\theta}_{g^\alpha_D,f}(x)&\geqslant -qw'_{\lambda_0} \quad\mbox{on\ }D,\label{D2}\end{align}
where\/ $w'_{\lambda_0}$ is the number determined by\/~{\rm(\ref{same})}.
\end{theorem}

\section{Variational inequalities for the $f$-weighted $\alpha$-Riesz potentials}\label{sec-main-var-riesz}
This section is devoted to necessary and/or sufficient conditions for the solvability of Problem~\ref{pr} with $\sigma\in\mathfrak C(F)\cup\{\infty\}$, given in terms of
variational inequalities for the $f$-weight\-ed $\alpha$-Riesz potentials. Throughout this section, we assume $D^c$ to be not $\alpha$-thin at~$\omega$.

Then, by Lemma~\ref{lemma-aux''}, for $\lambda^\sigma_{\mathbf A}=\lambda^+-\lambda^-$ to solve Problem~\ref{pr}, it is necessary and sufficient that $\lambda^+$ solve
Problem~\ref{pr2} with the same~$\sigma$. Furthermore, by~(\ref{reprrr}),
\begin{equation}\label{rel-bala}\lambda^-=\beta^\alpha_{D^c}\lambda^+,\end{equation}
which yields
\begin{equation*}\label{UW}W_{\alpha,f}^{\lambda^\sigma_{\mathbf A}}(x)=U_\alpha^{\lambda^+-\beta^\alpha_{D^c}\lambda^+}(x)+f(x)=
W_{g^\alpha_D,f}^{\lambda^+}(x)\quad\text{for all \ }x\in D.\end{equation*}

For the sake of simplicity, in the next assertion we assume that in the case $\alpha=2$, $D$ is simply connected.

\begin{lemma}\label{desc-riesz} If\/ $\lambda^\sigma_{\mathbf A}=\lambda^+-\lambda^-$ solves Problem\/~{\rm\ref{pr}}, then
\begin{equation}\label{lemma-desc-riesz}
 S^{\lambda^-}_{\mathbb R^n}=\left\{
\begin{array}{rll} D^c & \mbox{if} & \alpha<2,\\ \partial D  & \mbox{if} & \alpha=2.\\ \end{array} \right.
\end{equation}
\end{lemma}

Indeed, Lemma~\ref{desc-riesz} follows from (\ref{rel-bala}) and the description of the supports of the $\alpha$-Riesz balayage measures.

\subsection{Variational inequalities in the constrained $\alpha$-Riesz minimum energy problems}\label{sec-var1-Riesz}
In this section, consider $\xi\in\mathcal A(F)$ and assume $\xi(F_0)>1$, where $F_0$ was given by~(\ref{f0}). Combining what has been noticed just above with  the assertions of
Section~\ref{sec-var1} (for $\lambda^+$ instead of~$\lambda$ or~$\lambda_0$) results in the following Theorem~\ref{th-main'''} and Corollaries~\ref{cor-02'} and~\ref{cor-0'}.

\begin{theorem}\label{th-main'''}Let\/ Case\/~{\rm I} take place. Then\/ $\lambda^\xi_{\mathbf A}=\lambda^+-\lambda^-\in\mathcal E_{\alpha,f}^\xi(\mathbf A,\mathbf 1)$
is the\/ {\rm(}unique\/{\rm)} solution to Problem\/~{\rm\ref{pr}} if and only if\/~{\rm(\ref{rel-bala})} holds and, in addition, there exists\/ $w_{\lambda^\xi_{\mathbf
A}}\in\mathbb R$ possessing the following two properties:
\begin{align*}W^{\lambda^\xi_{\mathbf A}}_{\alpha,f}(x)&\geqslant w_{\lambda^\xi_{\mathbf A}}\quad(\xi-\lambda^+)\mbox{-a.e.~in\ }F,\\
W^{\lambda^\xi_{\mathbf A}}_{\alpha,f}(x)&\leqslant w_{\lambda^\xi_{\mathbf A}}\quad\mbox{for all \ }x\in S^{\lambda^+}_{D}.\end{align*}
\end{theorem}

\begin{corollary}\label{cor-02'} Assume that\/ $f\bigl|_D=U^\chi_{g^\alpha_D}$ for some\/ $\chi\in\mathfrak M^+(D)$. If\/ $\lambda^\xi_{\mathbf A}=\lambda^+-\lambda^-$ solves
Problem\/~{\rm\ref{pr}}, then\/ $C_\alpha\bigl(\partial D\cap S^{\xi-\lambda^+}_{\mathbb R^n}\bigr)=0$.
\end{corollary}

\begin{corollary}\label{cor-0'} Let\/ $f=0$. A measure\/ $\lambda^\xi_{\mathbf A}=\lambda^+-\lambda^-\in\mathcal E_{\alpha,f}^\xi(\mathbf A,\mathbf 1)$
solves Problem\/~{\rm\ref{pr}} if and only if\/ {\rm(\ref{rel-bala})} holds and, in addition, there exists a\/ {\rm(}unique\/{\rm)} number\/ $w'_{\lambda^\xi_{\mathbf
A}}\in(0,\infty)$ such that
\begin{align*}U^{\lambda^\xi_{\mathbf A}}_{\alpha}(x)&=w'_{\lambda^\xi_{\mathbf A}}\quad(\xi-\lambda^+)\mbox{-a.e.~in\ }F,\\
U^{\lambda^\xi_{\mathbf A}}_{\alpha}(x)&\leqslant w'_{\lambda^\xi_{\mathbf A}}\quad\mbox{for all \ }x\in D.\end{align*}
Furthermore, if\/ $\alpha<2$, then also\/ $S^{\lambda^+}_D=F$ and\/ $S^{\lambda^-}_{\mathbb R^n}=D^c$.
\end{corollary}

In the notations of Corollary~\ref{cor-0} with $\lambda^+$ in place of $\lambda_0$,
\[w'_{\lambda^\xi_{\mathbf A}}=\frac{E_{\alpha}(\lambda^\xi_{\mathbf
A},\xi-\lambda^+)}{(\xi-\lambda^+)(D)}=\frac{E_{g^\alpha_D}(\lambda^+,\xi-\lambda^+)}{(\xi-\lambda^+)(D)}=w'_{\lambda^+}.\]

\subsection{Variational inequalities in the unconstrained $\alpha$-Riesz minimum energy problems}\label{sec-var2-riesz}
In this section, $\sigma=\infty$. Similarly as it has been done just above, we derive the following corollaries from the assertions of Section~\ref{sec-var2}.

\begin{corollary}\label{cor-infty-1}Assume Case\/~{\rm I} takes place. A measure\/ $\lambda_{\mathbf A}=\lambda^+-\lambda^-\in\mathcal E_{\alpha,f}(\mathbf A,\mathbf 1)$ solves
Problem\/~{\rm\ref{pr}} if and only if\/
{\rm(\ref{rel-bala})} holds and, in addition, there exists a\/ {\rm(}unique\/{\rm)} number\/ $w_f'\in\mathbb R$ possessing the properties
\begin{align*}W^{\lambda_{\mathbf A}}_{\alpha,f}(x)&\geqslant w_f'\quad\text{n.e.~in\ }F,\\
W^{\lambda_{\mathbf A}}_{\alpha,f}(x)&\leqslant w_f'\quad\text{for all \ }x\in S^{\lambda^+}_D.\end{align*}
Furthermore, then\/ $w_f'=w_f$, where\/ $w_f$ is the number from Theorem\/~{\rm\ref{th-infty}}, and  assertions\/~{\rm(a)} and\/~{\rm(b)} of Corollary\/~{\rm\ref{cor-infty}}
both hold.\end{corollary}

For the sake of simplicity, in the following assertion we assume that in the case $\alpha=2$, $D\setminus F$ is simply connected.

\begin{corollary}\label{cor-infty-2}Let\/ $f=0$. A measure\/ $\lambda_{\mathbf A}=\lambda^+-\lambda^-\in\mathcal E_\alpha(\mathbf A,\mathbf 1)$ solves Problem\/~{\rm\ref{pr}} if
and only if\/
{\rm(\ref{rel-bala})} holds and there exists a\/ {\rm(}unique\/{\rm)} number\/ $w'\in(0,\infty)$ such that
\begin{align*}U^{\lambda_{\mathbf A}}_{\alpha}(x)&=w'\quad\mbox{n.e.~in \ }F,\\
U^{\lambda_{\mathbf A}}_{\alpha}(x)&\leqslant w'\quad\mbox{for all \ }x\in D,\\
U^{\lambda_{\mathbf A}}_{\alpha}(x)&< w'\quad\mbox{for all \ }x\in D\setminus\breve{F}.\end{align*}
Furthermore, then\/ $w'=w$, where\/ $w$ is the number from Corollary\/~{\rm\ref{cor-0'-infty}}, i.e.
\[w'=E_\alpha(\lambda_{\mathbf A},\lambda^+)=E_\alpha(\lambda_{\mathbf A})=E_{g^\alpha_D}(\lambda^+)=w_{g^\alpha_D}(F)=\bigl[C{g^\alpha_D}(F)\bigr]^{-1}=E_\alpha(\mathbf
A,\mathbf 1).\]
The descriptions of\/ $S^{\lambda^+}_D$ and\/ $S^{\lambda^-}_{\mathbb R^n}$ are given by~{\rm(\ref{desc-sup})} for\/~$\lambda^+$ in place of\/~$\lambda_F$ and\/
{\rm(\ref{lemma-desc-riesz})}, respectively.
\end{corollary}

\section{Proof of Theorem~\ref{silv-bound}\label{silv-bound-proof}}
Consider an exhaustion of $F$ by an increasing sequence of compact sets $K_k$, $k\in\mathbb N$. Since the constraint~$\xi$ is bounded, it holds
\begin{equation}\label{bbb}\lim_{k\to\infty}\,\xi(F\setminus K_k)=0.\end{equation}

Because of assumption~(\ref{green}), there exists $\{\mu_\ell\}_{\ell\in\mathbb N}\subset\mathcal E_{g^\alpha_D,f}^\xi(F,1)$ such that
\begin{equation}\label{bbb1}\lim_{\ell\to\infty}\,G_{g^\alpha_D,f}(\mu_\ell)=G_{g^\alpha_D,f}^\xi(F,1).\end{equation}
This sequence $\{\mu_\ell\}_{\ell\in\mathbb N}$ is vaguely bounded; hence, by~\cite[Chapter~III, Section~2, Prop.~9]{B2}, it has a
vague cluster point~$\mu_0$. We assert that, in~Case~I, $\mu_0$ is a solution to Problem~\ref{pr2}.

Since $\mathfrak M^+(F)$ is a vaguely closed subset of~$\mathfrak M^+(D)$, we get $\mu_0\in\mathfrak M^+(F)$ and $\mu_0\leqslant\xi$.
Let, further, $\{\mu_{\ell_m}\}_{m\in\mathbb N}$ be a subsequence of $\{\mu_\ell\}_{\ell\in\mathbb N}$ converging vaguely to~$\mu_0$. Then
\begin{align*}1=\lim_{m\to\infty}\,\mu_{\ell_m}(F)\geqslant\mu_0(F)&=\lim_{k\to\infty}\,\mu_0(K_k)\\
&{}\geqslant\lim_{k\to\infty}\,\limsup_{m\to\infty}\,\mu_{\ell_m}(K_k)=1-\lim_{k\to\infty}\,\liminf_{m\to\infty}\,\mu_{\ell_m}(F\setminus K_k).\end{align*}
On account of the fact that $\mu_{\ell_m}(F\setminus K_k)\leqslant\xi(F\setminus K_k)$ for all $m,\,k\in\mathbb N$, combining the preceding chain of inequalities
with~(\ref{bbb}) yields $\mu_0(F)=1$.

To complete the proof, it thus remains to observe that $G_{g^\alpha_D,f}(\mu_0)\leqslant G_{g^\alpha_D,f}^\xi(F,1)$, which is seen from~(\ref{bbb1}) in view of the lower
semicontinuity of~$G_{g^\alpha_D,f}$ on~$\mathcal E^+_{g^\alpha_D}(D)$ (see~Lemma~\ref{lemma-cont}).\qed

\section{Proof of Theorem~\ref{infcap}\label{infcap-proof}}
Under the assumptions of the theorem, Case~II with $\zeta\geqslant0$ takes place, and therefore
\begin{equation}\label{strpos}G_{g^\alpha_D,f}(\nu)=\|\nu\|^2_{g^\alpha_D}+2E_{g^\alpha_D}(\zeta,\nu)\geqslant E_{g^\alpha_D}(\nu)\geqslant0\quad\text{for all \ }\nu\in\mathcal
E^+_{g^\alpha_D}(D).\end{equation}

Consider an exhaustion of $F$ by an increasing sequence of compact sets $K_k$, $k\in\mathbb N$. Since $C_{g^\alpha_D}(F)=\infty$, it holds $C_{g^\alpha_D}(F\setminus
K_k)=\infty$ for all $k\in\mathbb N$. Hence, for every~$k$ one can choose a measure
$\nu_k\in\mathcal E_{g^\alpha_D}^+(F\setminus K_k,1)$ with compact support so that
\begin{equation}\label{000}\lim_{k\to\infty}\,\|\nu_k\|^2_{g^\alpha_D}=0.\end{equation}
Certainly, there is no loss of generality in assuming $K_k\cup S_D^{\nu_k}\subset K_{k+1}$.

Fix $\xi\in\mathfrak C(F)$ and write $\xi_0:=\xi+\sum_{k\in\mathbb N}\,\nu_k$; then $\xi_0\in\mathfrak C(F)\setminus\mathfrak C_0(F)$. Due to~(\ref{identtt}), for each
$\sigma\in\mathfrak C(F)\cup\{\infty\}$ such that $\sigma\geqslant\xi_0$, it holds
\begin{equation*}\nu_k\in\mathcal E^\sigma_{g^\alpha_D,f}(F,1)\quad\text{for all \ }k\in\mathbb N.\end{equation*}
From (\ref{000}) and the Cauchy--Schwarz inequality we get
\begin{equation*}\label{111''}\lim_{k\to\infty}\,G_{g^\alpha_D,f}(\nu_k)=\lim_{k\to\infty}\,\bigl[\|\nu_k\|^2_{g^\alpha_D}+2E_{g^\alpha_D}(\zeta,\nu_k)\bigr]\leqslant
2\|\zeta\|_{g^\alpha_D}\lim_{k\to\infty}\,\|\nu_k\|_{g^\alpha_D}=0.\end{equation*}
Combined with~(\ref{strpos}), this yields $G^\sigma_{g^\alpha_D,f}(F,1)=0$. Repeated application of~(\ref{strpos}) shows also that such an infimum value can be attained only at
zero measure.
As $0\not\in\mathcal E^\sigma_{g^\alpha_D,f}(F,1)$, Problem~\ref{pr2} with $\sigma$ specified above is unsolvable.\qed

\section{$\alpha$-Green strong completeness theorem}\label{crucial''}

A  crucial point in our proof of Theorem~\ref{lemma-main}, given in Section~\ref{lemma-main-proof}, is the following perfectness-type result for the $\alpha$-Green kernel.

\begin{theorem}\label{alphaeq1} Let\/ $E\subset\mathbb R^n$ be closed in\/~$D$. If\/ $\alpha<2$, require additionally\/ $\overline{E}\cap\partial_{\overline{\mathbb R^n}}D$ to consist of at most one point. Then any strong Cauchy sequence\/ $\{\nu_k\}_{k\in\mathbb N}\subset\mathcal
E^+_{g_D^\alpha}(E)$ with
\begin{equation}\label{bound}\sup_{k\in\mathbb
N}\,\nu_k(E)<\infty\end{equation} converges both strongly and vaguely to
the unique\/ $\nu_0\in\mathcal E^+_{g_D^\alpha}(E)$.
\end{theorem}

\begin{proof} We can certainly assume that $\alpha<2$, since otherwise the theorem holds true due to the perfectness of the $g^2_D$-kernel, established in~\cite{E1}.

The (strongly fundamental) sequence $\{\nu_k\}_{k\in\mathbb N}\subset\mathcal
E^+_{g_D^\alpha}(E)$ is strongly
bounded, i.e. \begin{equation*}\label{strbound}\sup_{k\in\mathbb
N}\,\|\nu_k\|_{g_D^\alpha}<\infty.\end{equation*} Besides,
by~(\ref{bound}), $\{\nu_k\}_{k\in\mathbb N}$ is vaguely bounded, and hence it has a
vague cluster point~$\nu_0$. Since $\mathcal E^+_{g_D^\alpha}(E)$ is a vaguely closed subset of~$\mathfrak M^+(D)$, we have
$\nu_0\in\mathcal E^+_{g_D^\alpha}(E)$.

Let $\bigl\{\nu_{k_\ell}\bigr\}_{\ell\in\mathbb N}$ be a subsequence
of~$\{\nu_k\}_{k\in\mathbb N}$ such that
\begin{equation}\label{vague}\nu_{k_\ell}\to\nu_0\quad\text{vaguely as \
}\ell\to\infty.\end{equation}
We next proceed to show that $\nu_{k_\ell}\to\nu_0$ also strongly
in~$\mathcal E^+_{g_D^\alpha}(E)$, i.e.
\begin{equation}\label{strong}\lim_{\ell\to\infty}\,\|\nu_{k_\ell}-\nu_0\|_{g_D^\alpha}=0.\end{equation}

As $\bigl\{\nu_{k_\ell}\bigr\}_{\ell\in\mathbb N}\subset\mathcal
E^+_{g_D^\alpha}(E)$, being a subsequence of the strong Cauchy
sequence~$\{\nu_k\}_{k\in\mathbb N}$, is strongly fundamental as
well, we see from (\ref{EG}) that
\begin{equation}\label{form'}\widetilde{\nu_{k_\ell}}:=\nu_{k_\ell}-\beta^\alpha_{D^c}\nu_{k_\ell},\quad \ell\in\mathbb
N,\end{equation} is strongly fundamental in~$\mathcal E_\alpha(\mathbb R^n)$.
The proof of~(\ref{strong}) is given in two steps.

{\bf Step 1.} Throughout this step, let $\overline{E}\cap\partial_{\overline{\mathbb R^n}}D$ either be empty or consist of only~$\omega$. Then $\nu_{k_\ell}$
and~$\beta^\alpha_{D^c}\nu_{k_\ell}$ are supported by the sets~$E$ and~$D^c$, which due to the assumptions made are closed in~$\mathbb R^n$ and nonintersecting. Consider the
condenser $\mathbf B:=(E,D^c)$. The
strong completeness theorem from~\cite{ZUmzh} (or see Theorem~\ref{th-complete} above) yields that there exists the unique
measure $\tilde{\nu}=\tilde{\nu}^+-\tilde{\nu}^-\in\mathcal
E_\alpha(\mathbf B)$
such that
\begin{equation}\label{um'}\lim_{\ell\to\infty}\,\|\widetilde{\nu_{k_\ell}}-
\tilde{\nu}\|_\alpha=0.\end{equation} Furthermore, by this theorem,
$\tilde{\nu}^+$ and~$\tilde{\nu}^-$ are the vague limits of the
positive and the negative parts
of~$\widetilde{\nu_{k_\ell}}$, $\ell\in\mathbb N$,
respectively.  In view of~(\ref{vague}), we thus have
\begin{equation}\label{equal'}\tilde{\nu}^+=\nu_0,\end{equation}
for the vague topology is Hausdorff.

By the remark in~\cite[p.~166]{F1}, it follows from~(\ref{um'}) that there exists a subsequence of the sequence
$\bigl\{\widetilde{\nu_{k_\ell}}\bigr\}_{\ell\in\mathbb
N}$ (denote it again by the same symbol) such that
\[U_\alpha^{\tilde{\nu}}(x)=\lim_{\ell\to\infty}\,U_\alpha^{\widetilde{\nu_{k_\ell}}}(x)
\quad\text{n.e.~in \ }\mathbb R^n.\] On account of~(\ref{form'}) and the
countable subadditivity of $C_\alpha(\cdot)$ over Borel
sets,
we see from the preceding relation that $U_\alpha^{\tilde{\nu}}(x)=0$
n.e.~in~$D^c$ and, therefore, $\tilde{\nu}^-=\beta^\alpha_{D^c}\tilde{\nu}^+$.
Combining this with (\ref{form'}), (\ref{um'}) and~(\ref{equal'}) implies
\[\lim_{\ell\to\infty}\,\bigl\|\bigl(\nu_{k_\ell}-\nu_0\bigr)-
\beta^\alpha_{D^c}\bigl(\nu_{k_\ell}-\nu_0\bigr)\bigr\|_\alpha=0,\]
which in view of~(\ref{EG}) establishes (\ref{strong}).

{\bf Step 2.} We next prove relation~(\ref{strong}) in the case $\overline{E}\cap\partial_{\overline{\mathbb R^n}}D=\{x_0\}$, where $x_0\ne\omega$. Throughout this step, all the
measures can be assumed to have zero mass at~$x_0$, for we can restrict our consideration to those with finite energy.

Define the inversion with respect to~$S(x_0,1)$, namely, each point $x\ne
x_0$ is mapped to the point~$x^*$ on the ray through~$x$ which
issues from~$x_0$, determined uniquely by
\[|x-x_0|\cdot|x^*-x_0|=1.\]
This is a one-to-one, bicontinuous mapping of $\mathbb
R^n\setminus\{x_0\}$ onto itself; furthermore,
\begin{equation}\label{inv}|x^*-y^*|=\frac{|x-y|}{|x_0-x||x_0-y|}.\end{equation}
It can be extended to a one-to-one,
bicontinuous map of $\overline{\mathbb R^n}$ onto itself by setting $x_0\mapsto\omega$.

To each $\nu\in\mathfrak M(\mathbb R^n)$ (with
$\nu(\{x_0\})=0$) we correspond the Kelvin transform
$\nu^*\in\mathfrak M(\mathbb R^n)$ by means of the
formula
\[d\nu^*(x^*)=|x-x_0|^{\alpha-n}\,d\nu(x),\quad x^*\in\mathbb R^n.\]
Then, in view of~(\ref{inv}),
\begin{equation}\label{KP}U_\alpha^{\nu^*}(x^*)=|x-x_0|^{n-\alpha}U_\alpha^{\nu}(x),\quad x^*\in\mathbb R^n,\end{equation}
and therefore
\begin{equation}\label{K}E_\alpha(\nu^*)=E_\alpha(\nu)\end{equation}
(see \cite[Chapter IV, Section 5, n$^\circ$\,19]{L} and
\cite[Chapter V, Section 2, n$^\circ$\,8]{L}, respectively).

It is obvious that the Kelvin transformation is additive, i.e.
\begin{equation}\label{A}\bigl(\nu_1+\nu_2\bigr)^*=
\nu_1^*+\nu_2^*.\end{equation}

We also observe that
\begin{equation}\label{B}(\beta^\alpha_{D^c}\nu)^*=\beta^\alpha_{(\overline{D^c})^*}\nu^*,\end{equation} where $(\overline{D^c})^*$ is
the image of~$\overline{D^c}$ under the inversion $x\mapsto x^*$.\footnote{In fact, $(\overline{D^c})^*=(D^c)^*\cup\{x_0\}$.} Indeed, in view
of~(\ref{KP}) and the definition of the $\alpha$-Riesz balayage, we get
\[U_\alpha^{(\beta^\alpha_{D^c}\nu)^*}(x^*)=|x-x_0|^{n-\alpha}U_\alpha^{\beta^\alpha_{D^c}\nu}(x)
=|x-x_0|^{n-\alpha}U_\alpha^\nu(x)=U_\alpha^{\nu^*}(x^*),\] the
relation being valid for nearly all $x\in D^c$. Consequently, it also holds for nearly all
$x^*\in(\overline{D^c})^*$, because the inversion of a set with $C_\alpha(\cdot)=0$
has the interior $\alpha$-Riesz capacity zero as well
(see~\cite[Chapter~IV, Section~5, n$^\circ$\,19]{L}). Since $(\beta^\alpha_{D^c}\nu)^*$ is supported by~$(\overline{D^c})^*$, identity~(\ref{B}) follows.

 Applying \cite[Lemma~4.3]{L} to
$\nu_{k_\ell}$, $\ell\in\mathbb N$, and~$\nu_0$ (where $\nu_{k_\ell}$, $\ell\in\mathbb N$, and~$\nu_0$ are as above), on account
of~(\ref{bound}) and~(\ref{vague}) we have
\begin{equation}\label{vague'}\nu_{k_\ell}^*\to\nu_0^*\quad\text{vaguely as \
}\ell\to\infty.\end{equation}
Also observe that, according to~(\ref{K}) and the fact that
$\{\widetilde{\nu_{k_\ell}}\}_{\ell\in\mathbb N}$ is strongly fundamental, so is the sequence
$\bigl(\widetilde{\nu_{k_\ell}}\bigr)^*\in\mathcal E_\alpha(\mathbb R^n)$, $\ell\in\mathbb N$, which in consequence of~(\ref{form'}), (\ref{A}) and~(\ref{B}) can be rewritten in
the form
\begin{equation}\label{form}\bigl(\widetilde{\nu_{k_\ell}}\bigr)^*=\nu_{k_\ell}^*-
\bigl(\beta^\alpha_{D^c}\nu_{k_\ell}\bigr)^*=
\nu_{k_\ell}^*-\beta^\alpha_{(\overline{D^c})^*}\nu_{k_\ell}^*,\quad
\ell\in\mathbb N.\end{equation} The positive and the negative parts
of~$\bigl(\widetilde{\nu_{k_\ell}}\bigr)^*$ are supported by the
sets~$E^*$ and~$(\overline{D^c})^*$, respectively,
which are closed in~$\mathbb R^n$ and nonintersecting; hence, the
strong completeness theorem from~\cite{ZUmzh} (see Theorem~\ref{th-complete} above) can be applied. Therefore, there exists the unique
measure $\hat{\nu}=\hat{\nu}^+-\hat{\nu}^-\in\mathcal
E_\alpha(\mathbb R^n)$, where $\hat{\nu}^+$ and~$\hat{\nu}^-$ are supported
by~$E^*$ and~$(\overline{D^c})^*$, respectively,
such that
\begin{equation}\label{um}\lim_{\ell\to\infty}\,\bigl\|\bigl(\widetilde{\nu_{k_\ell}}\bigr)^*-
\hat{\nu}\bigr\|_\alpha=0.\end{equation} Furthermore,
$\hat{\nu}^+$ and~$\hat{\nu}^-$ are the vague limits of the
positive and the negative parts
of~$\bigl(\widetilde{\nu_{k_\ell}}\bigr)^*$, $\ell\in\mathbb N$,
respectively. When combined with~(\ref{vague'}), (\ref{form}) and the fact that the vague topology is Hausdorff, this
implies
\begin{equation}\label{equal}\hat{\nu}^+=\nu^*_0.\end{equation}

In view of (\ref{um}) and the remark in~\cite[p.~166]{F1},
one can choose a subsequence of the sequence
$\bigl\{\bigl(\widetilde{\nu_{k_\ell}}\bigr)^*\bigr\}_{\ell\in\mathbb
N}$ (denote it again by the same symbol) so that
\[U_\alpha^{\hat{\nu}}(x)=\lim_{\ell\to\infty}\,U_\alpha^{(\widetilde{\nu_{k_\ell}})^*}(x)
\quad\text{n.e.~in \ }\mathbb R^n.\] On account of~(\ref{form}),
we thus have $U_\alpha^{\hat{\nu}}(x)=0$
n.e.~in~$(\overline{D^c})^*$, and therefore, by~(\ref{equal}),
\begin{equation}\label{umm}\hat{\nu}^-=\beta^\alpha_{(\overline{D^c})^*}\hat{\nu}^+=
\beta^\alpha_{(\overline{D^c})^*}\nu_0^*.\end{equation}

Using the fact that the Kelvin transformation is an involution and
applying~(\ref{K}), (\ref{A}) and~(\ref{B}) again, we conclude
from~(\ref{um}), (\ref{equal}) and~(\ref{umm}) that
\[\widetilde{\nu_{k_\ell}}\to\nu_0-\beta^\alpha_{D^c}\nu_0
\quad\text{(as \ }\ell\to\infty)\quad\text{in \ }\mathcal
E_\alpha(\mathbb R^n),\] or equivalently, by the definition
of~$\widetilde{\nu_{k_\ell}}$,
\[\lim_{\ell\to\infty}\,\bigl\|\bigl(\nu_{k_\ell}-\nu_0\bigr)-
\beta^\alpha_{D^c}\bigl(\nu_{k_\ell}-\nu_0\bigr)\bigr\|_\alpha=0.\]
Repeated application of~(\ref{EG}) then proves relation~(\ref{strong}) also in the case $\overline{E}\cap\partial_{\overline{\mathbb R^n}}D=\{x_0\}$, where $x_0\ne\omega$. This
completes Step~2.

Since the sequence $\{\nu_k\}_{k\in\mathbb N}$ is strongly
fundamental, $\nu_k\to\nu_0$ strongly by~(\ref{strong}).
It has thus been proved that $\{\nu_k\}_{k\in\mathbb N}$ converges strongly to any of its vague cluster points. As the $\alpha$-Green kernel is strictly positive
definite, $\nu_0$ is the only vague cluster point of $\{\nu_k\}_{k\in\mathbb N}$, and so $\nu_k\to\nu_0$ also vaguely.
\end{proof}

\section{$\alpha$-Green equilibrium measure}\label{crucial}

\begin{theorem}\label{alphaeq} Let\/ $E\subset\mathbb R^n$ be closed in\/~$D$. If \/$\alpha<2$, require additionally\/ $\overline{E}\cap\partial_{\overline{\mathbb R^n}}D$ to consist of at most one
point.
If, moreover, $C_{g_D^\alpha}(E)<\infty$, then there
exists an\/ $\alpha$-Gre\-en interior equilibrium
measure\/~$\gamma=\gamma_E$ on\/~$E$, that is,  a one possessing the
properties\/ $\gamma\in\mathcal E_{g^\alpha_D}^+(E)$ and
\begin{align}\label{equ.00'''}E_{g^\alpha_D}(\gamma)&=\gamma(E)=C_{g^\alpha_D}(E),\\
\label{equ.10'''}U_{g^\alpha_D}^{\gamma}(x)&\geqslant1\quad\text{n.e.~in
\ }E,\\
\label{equ.100'''}U_{g^\alpha_D}^{\gamma}(x)&\leqslant1\quad\text{for
all \ }x\in S^\gamma_D.\end{align} This\/~$\gamma$ solves the problem
of minimizing the energy\/~$E_{g^\alpha_D}(\nu)$ over the convex
class\/~$\Gamma_E$ of all\/~$\nu\in\mathcal E_{g^\alpha_D}(D)$ such
that\/ $U_{g^\alpha_D}^{\nu}(x)\geqslant1$ n.e.~in\/~$E$, and hence
it is unique.
\end{theorem}

\begin{proof}This theorem needs to be proved only in the case $\alpha<2$, since otherwise it is a special case of~\cite[Theorem~4.1]{F1} in view of the perfectness of the
$g^2_D$-Green kernel.  Also note that we can assume~$E$ to be noncompact in~$D$, for if not, then the theorem follows from~\cite{F1} (see Theorem~2.5 and Lemma~3.2.2 with $t=1$
therein).

Consider an exhaustion of~$E$ by an increasing sequence of sets $K_k\subset E$, $k\in\mathbb N$,
compact in~$D$, and let $\gamma_k=\gamma_{K_k}$ be the $\alpha$-Green equilibrium measure
on~$K_k$.
Then, by~(\ref{equ.00'''}) with $E=K_k$ and (\ref{ex}),
\begin{equation}\label{lim}\lim_{k\to\infty}\,\|\gamma_k\|_{g_D^\alpha}^2=\lim_{k\to\infty}\,C_{g_D^\alpha}(K_k)=C_{g_D^\alpha}(E)<\infty.\end{equation}
Since $\gamma_k\in\Gamma_{K_p}$ for all $k\geqslant p$, which is seen from the monotonicity of $\{K_k\}_{k\in\mathbb N}$ and inequality~(\ref{equ.10'''}) with $E=K_k$,
Lemma~\ref{convex} yields
\[\|\gamma_k-\gamma_p\|_{g_D^\alpha}^2\leqslant\|\gamma_k\|_{g_D^\alpha}^2-\|\gamma_p\|_{g_D^\alpha}^2\quad\text{for all \ }k\geqslant p.\]
In consequence of the last two relations,
$\{\gamma_k\}_{k\in\mathbb N}\subset\mathcal E^+_{g_D^\alpha}(E)$ is
strongly fundamental. In addition, by~(\ref{equ.00'''}) with $E=K_k$,
\begin{equation}\label{hhhh}\gamma_k(E)=C_{g_D^\alpha}(K_k)<
C_{g_D^\alpha}(E)<\infty\quad\text{for all \ }k\in\mathbb N,\end{equation} so
that all the assumptions of Theorem~\ref{alphaeq1}
for~$\{\gamma_k\}_{k\in\mathbb N}$ are satisfied. Hence, there exists the
unique $\gamma\in\mathcal E^+_{g_D^\alpha}(E)$ such that
$\gamma_k\to\gamma$ both strongly and vaguely.

On account of~(\ref{lim}), we thus get
\begin{equation}\label{hh}\|\gamma\|_{g_D^\alpha}^2=\lim_{k\to\infty}\,\|\gamma_k\|_{g_D^\alpha}^2=C_{g_D^\alpha}(E).\end{equation}
According to \cite{F1} (see the remark on p.~166 therein), the strong
convergence of $\gamma_k$ to~$\gamma$ also yields that there exists
a subsequence $\gamma_{k_\ell}=\gamma_{K_{k_\ell}}$, $\ell\in\mathbb N$, of $\{\gamma_k\}_{k\in\mathbb N}$ such that
\begin{equation*}\lim_{\ell\to\infty}\,U_{g_D^\alpha}^{\gamma_{k_\ell}}(x)=
U_{g_D^\alpha}^\gamma(x)\quad\text{n.e.~in \ }D,
\end{equation*}  while by (\ref{equ.10'''}) for $E=K_{k_\ell}$,
\begin{equation*}U_{g_D^\alpha}^{\gamma_{k_\ell}}(x)\geqslant1\quad\text{n.e.~in \ }K_{k_\ell}.\end{equation*}
Since the sets $K_{k_\ell}$, $\ell\in\mathbb N$, increase and $E=\bigcup_{\ell\in\mathbb N}K_{k_\ell}$, the last two relations imply~(\ref{equ.10'''}).
Here we have used the fact that the $\alpha$-Green capacity of a countable union
of Borel sets with zero $\alpha$-Green capacity is still zero; see~\cite{F1}.

Fix $x\in S^\gamma_D$. As $\gamma_k\to\gamma$ vaguely, one can choose
$x_k\in S^{\gamma_k}_D$ so that $x_k\to x$ as $k\to\infty$. Because of the
fact (cf.~\cite[Lemma~2.2.1]{F1}) that $U_{g_D^\alpha}^{\mu}(x)$ is
lower semicontinuous on the product space $D\times\mathfrak M^+(D)$,
where $\mathfrak M^+(D)$ is equipped with the vague topology, we get
\[U_{g_D^\alpha}^\gamma(x)\leqslant\liminf_{k\to\infty}\,U_{g_D^\alpha}^{\gamma_k}(x_k).\]
Since, by (\ref{equ.100'''}) for $E=K_k$, $U_{g_D^\alpha}^{\gamma_k}(x_k)\leqslant1$ for all $k\in\mathbb N$,
inequality~(\ref{equ.100'''}) follows.

In view of the vague convergence of $\gamma_k$ to $\gamma$, we also have
\begin{equation*}\gamma(E)\leqslant\liminf_{k\to\infty}\,\gamma_k(E),\end{equation*}
so that
$\gamma(E)\leqslant C_{g_D^\alpha}(E)$ by~(\ref{hhhh}). When combined with~(\ref{hh}), this shows that, in order to complete the proof
of~(\ref{equ.00'''}), it is left to establish the inequality $\gamma(E)\geqslant C_{g_D^\alpha}(E)$, but it follows at once by integrating~(\ref{equ.100'''}) with respect
to~$\gamma$.

Finally, \cite[Lemma~3.2.2]{F1} with $t=1$ yields the very last assertion of the theorem.\end{proof}

\begin{remark} $\gamma_E$ coincides up to a constant factor with the solution~$\lambda_E$ of Problem~\ref{pr2} (for $E$ in place of~$F$) with
$\sigma=\infty$ and $f=0$. See Corollaries~\ref{cor-0'-infty} and~\ref{cor-sup} for a more detailed information about the properties of the $\alpha$-Green
potential and the support of~$\lambda_E$.\end{remark}

\section{Proof of Theorem \ref{lemma-main}}\label{lemma-main-proof}
In this section we follow methods developed in~\cite{ZCAOT}
(see Theorems~2.2 and~3.1 therein).  Under the assumptions of
Theorem~\ref{lemma-main}, the following auxiliary result holds.

\begin{lemma}\label{lemma-aux} For any\/ $\sigma\in\mathfrak C(F)\cup\{\infty\}$, the metric space
\[\mathcal E^\sigma_{g^\alpha_D}(F,1):=
\bigl\{\mu\in\mathcal E^+_{g^\alpha_D}(F,1): \
\mu\leqslant\sigma\bigr\}\] is strongly complete. In more detail, any strong Cauchy sequence\/ $\{\mu_k\}_{k\in\mathbb
N}\subset\mathcal E^\sigma_{g^\alpha_D}(F,1)$ converges both
strongly and vaguely to the unique\/ $\mu_0\in\mathcal
E^\sigma_{g^\alpha_D}(F,1)$.
\end{lemma}

\begin{proof}Fix a strong Cauchy sequence $\{\mu_k\}_{k\in\mathbb N}\subset\mathcal
E^\sigma_{g^\alpha_D}(F,1)$. According to Theorem~\ref{alphaeq1}, there exists the unique $\mu_0\in\mathcal E^+_{g^\alpha_D}(F)$
such that
\[\mu_k\to\mu_0\quad\text{strongly and
vaguely}.\]  Actually, $\mu_0\in\mathcal
E^\sigma_{g^\alpha_D}(F)$, since $\mathcal E_{g^\alpha_D}^\sigma(F)$
is vaguely closed. Hence, it is left to show that
\begin{equation}\label{b}\mu_0(F)=1.\end{equation}

Assume $F$ to be noncompact, for if not, then (\ref{b}) is evident. Consider an exhaustion of~$F$ by an increasing sequence of sets $K_m\subset F$, $m\in\mathbb N$,
compact in~$D$; then
\begin{align*}1=\lim_{k\to\infty}\,\mu_k(F)\geqslant\mu_0(F)&=\lim_{m\to\infty}\,\mu_0(K_m)\geqslant\lim_{m\to\infty}\,\limsup_{k\to\infty}\,\mu_k(K_m)\\&{}=
1-\lim_{m\to\infty}\,\liminf_{k\to\infty}\,\mu_k(F\setminus
K_m).\end{align*} Therefore, identity~(\ref{b}) will be established once we prove
\begin{equation}\label{l}\lim_{m\to\infty}\,\liminf_{k\to\infty}\,\mu_k(F\setminus
K_m)=0.\end{equation}

Write $K_m^*:={\mathrm C\ell}_D(F\setminus K_m)$. It is seen from Theorem~\ref{alphaeq} that, under the assumptions made, there exists
the $\alpha$-Green
equilibrium measure~$\gamma_m$ on~$K_m^*$, and it solves the problem of minimizing $E_{g^\alpha_D}(\nu)$ over the convex cone $\Gamma_{K_m^*}$. Since, by the monotonicity
of~$K_m^*$, $m\in\mathbb N$, and relation~(\ref{equ.10'''}) for~$E=K_m^*$, $\gamma_m$ belongs to~$\Gamma_p$ for all $p\geqslant m$, Lemma~\ref{convex}
yields
\[
\|\gamma_m-\gamma_p\|^2_{g_D^\alpha}\leqslant\|\gamma_m\|^2_{g_D^\alpha}-\|\gamma_p\|^2_{g_D^\alpha}\quad\mbox{for all \ }p\geqslant m.\]
Furthermore, it is clear from~(\ref{equ.00'''}) for $E=K_m^*$ that the sequence
$\|\gamma_m\|^2_{g_D^\alpha}$, $m\in\mathbb N$, is
bounded and nonincreasing, and hence it is fundamental in~$\mathbb R$. The
preceding inequality thus implies that $\gamma_m$,
$m\in\mathbb N$, is strongly fundamental in~$\mathcal
E^+_{g_D^\alpha}(D)$. Since it obviously converges vaguely to zero, zero is
also its strong limit due to Theorem~\ref{alphaeq1}. Hence,
\begin{equation*}
\lim_{m\to\infty}\,\|\gamma_m\|_{g_D^\alpha}=0.
\label{27}
\end{equation*}
Besides, by~(\ref{equ.10'''}) for $E=K^*_m$,
\[\mu_k(F\setminus K_m)\leqslant\mu_k(K^*_m)\leqslant\bigl\langle
U_{g^\alpha_D}^{\gamma_m},\mu_k\bigr\rangle\leqslant\|\gamma_m\|_{g^\alpha_D}\cdot\|\mu_k\|_{g^\alpha_D}\quad\text{for all
\ }k,\,m\in\mathbb N.\] As $\|\mu_k\|_{g^\alpha_D}$,
$k\in\mathbb N$, is bounded, combining the last two relations yields~(\ref{l}).\end{proof}

Now we are able to complete the proof of Theorem~\ref{lemma-main}. In view of~(\ref{green}), one can choose
$\nu_k\in\mathcal E^\sigma_{g^\alpha_D,f}(F,1)$, $k\in\mathbb N$,
so that
\begin{equation}\label{hryu'}\lim_{k\to\infty}\,G_{g^\alpha_D,f}(\nu_k)=G^\sigma_{g^\alpha_D,f}(F,1)<\infty.\end{equation}
Based on the convexity of the class
$\mathcal E^\sigma_{g^\alpha_D,f}(F,1)$ and the pre-Hilbert structure on~$\mathcal E_{g^\alpha_D}(D)$, with the help of arguments similar to those in the proof of
Lemma~\ref{uniqueness} we obtain
\begin{equation*}\label{par}0\leqslant\|\nu_k-\nu_p\|^2_{g^\alpha_D}\leqslant-4G^\sigma_{g^\alpha_D,f}(F,1)+
2G_{g^\alpha_D,f}(\nu_k)+2G_{g^\alpha_D,f}(\nu_p)\quad\text{for all \ }k,p\in\mathbb N.\end{equation*}
Substituting (\ref{hryu'}) into this relation implies that $\{\nu_k\}_{k\in\mathbb N}$ is strongly
fundamental in the metric space $\mathcal
E^\sigma_{g^\alpha_D}(F,1)$. By Lemma~\ref{lemma-aux},
$\{\nu_k\}_{k\in\mathbb N}$ therefore converges both strongly and vaguely to the unique
$\nu_0\in\mathcal E^\sigma_{g^\alpha_D}(F,1)$. On account of Lemma~\ref{lemma-cont}, we thus get
\begin{equation}\label{111}G_{g^\alpha_D,f}(\nu_0)\leqslant\lim_{k\to\infty}\,G_{g^\alpha_D,f}(\nu_k)=
G^\sigma_{g^\alpha_D,f}(F,1)<\infty.\end{equation} Hence, $\nu_0\in\mathcal E^\sigma_{g^\alpha_D,f}(F,1)$ and, consequently, $G_{g^\alpha_D,f}(\nu_0)\geqslant
G^\sigma_{g^\alpha_D,f}(F,1)$. Combined with~(\ref{111}), this shows that $\nu_0=:\lambda^\sigma_{F}$ is the solution to Problem~\ref{pr2}.\qed

\section{Proof of the assertions formulated in Section~\ref{sec-main-var-green}}\label{ss-1}
\subsection{Proof of Theorem \ref{th-main''}}\label{th-main''-proof}

Fix $\lambda\in\mathcal E_{g^\alpha_D,f}^\xi(F,1)$, and first assume that it solves Problem~\ref{pr2}. Then inequality~(\ref{b1}) holds for $w_\lambda=L$, where
\[L:=\sup\,\bigl\{q\in\mathbb R: \ W_{g^\alpha_D,f}^\lambda(x)\geqslant q\quad(\xi-\lambda)\mbox{-a.e.~in\ }F\bigr\}.\] In turn,
(\ref{b1}) with $w_\lambda=L$ implies $L<\infty$, since
$W_{g^\alpha_D,f}^\lambda(x)<\infty$ holds n.e.~in~$F_0$, hence $(\xi-\lambda)$-a.e.~in~$F_0$, while
$(\xi-\lambda)(F_0)>0$. Also, $L>-\infty$, for
 $f$ is bounded from below.

We proceed by establishing (\ref{b2})
for $w_\lambda=L$. Having denoted (cf.~\cite{DS,R})
\[F^+(w):=\bigl\{x\in F:\ W_{g^\alpha_D,f}^\lambda(x)>w\bigr\}\quad\text{and}\quad
F^-(w):=\bigl\{x\in F:\ W_{g^\alpha_D,f}^\lambda(x)<w\bigr\},\]
where $w\in\mathbb R$ is arbitrary, we assume on the contrary that
(\ref{b2}) for $w_\lambda=L$ does not hold. In view of the lower semicontinuity of~$W_{g^\alpha_D,f}^\lambda$ on~$F$, then one can choose $w_1\in(L,\infty)$ so that
$\lambda\bigl(F^+(w_1)\bigr)>0$. At the same time, as $w_1>L$, relation~(\ref{b1})  with $w_\lambda=L$  yields
$(\xi-\lambda)\bigl(F^-(w_1)\bigr)>0$. Therefore, there exist compact sets $K_1\subset F^+(w_1)$ and $K_2\subset F^-(w_1)$ such that
\[0<\lambda(K_1)<(\xi-\lambda)(K_2).\]

Write $\tau:=(\xi-\lambda)\bigl|_{K_2}$; then $\tau\in\mathcal E^+_{g^\alpha_D}(K_2)$. Since $\bigl\langle W_{g^\alpha_D,f}^\lambda,\tau\bigr\rangle\leqslant
w_1\tau(K_2)<\infty$, we get $\langle f,\tau\rangle<\infty$. Define
\[\theta:=\lambda-\lambda\bigl|_{K_1}+c\tau,\quad\mbox{where
\ }c:=\lambda(K_1)\bigl/\tau(K_2)\in(0,1).\]
A straightforward verification shows that $\theta(F)=1$ and $\theta\leqslant\xi$, and so $\theta\in\mathcal E^\xi_{g^\alpha_D,f}(F,1)$. On the other hand,
\begin{align*}
\bigl\langle W_{g^\alpha_D,f}^\lambda,\theta-\lambda\bigr\rangle&=\bigl\langle
W_{g^\alpha_D,f}^\lambda-w_1,\theta-\lambda\bigr\rangle\\&{}=-\bigl\langle
W_{g^\alpha_D,f}^\lambda-w_1,\lambda\bigl|_{K_1}\bigr\rangle+c\bigl\langle
W_{g^\alpha_D,f}^\lambda-w_1,\tau\bigr\rangle<0,\end{align*} which is
impossible in view of Lemma~\ref{lequiv}. This proves the necessary part of
the theorem.

Next, let $\lambda$ satisfy both~(\ref{b1}) and~(\ref{b2}) for some
$w_\lambda\in\mathbb R$. Then $\lambda\bigl(F^+(w_\lambda)\bigr)=0$ and
$\bigl(\xi-\lambda\bigr)\bigl(F^-(w_\lambda)\bigr)=0$. For any $\nu\in\mathcal E^\xi_{g^\alpha_D,f}(F,1)$, we therefore obtain
\begin{eqnarray*}\bigl\langle
W_{g^\alpha_D,f}^\lambda,\nu-\lambda\bigr\rangle&\!\!\!\!\!=\!\!\!\!\!&\bigl\langle
W_{g^\alpha_D,f}^\lambda-w_\lambda,\nu-\lambda\bigr\rangle\\
&\!\!\!\!\!{}=\!\!\!\!\!&\bigl\langle W_{g^\alpha_D,f}^\lambda-w_\lambda,\nu\bigl|_{F^+(w_\lambda)}\bigr\rangle+\bigl\langle
W_{g^\alpha_D,f}^\lambda-w_\lambda,(\nu-\xi)\bigl|_{F^-(w_\lambda)}\bigr\rangle\geqslant0.\end{eqnarray*}
Application of ~Lemma~\ref{lequiv} shows that, indeed, $\lambda$ is the solution to Problem~\ref{pr2}.\qed

\subsection{Proof of Corollary \ref{cor-02}}\label{proof1} Under the conditions of the corollary, $W_{g^\alpha_D,f}^\lambda(x)>0$ in~$D$; hence, the number~$w_\lambda$ from Theorem~\ref{th-main''} satisfies relation
\begin{equation}\label{strposss}w_\lambda\in(0,\infty).\end{equation} Furthermore, in the notation $\Psi(x):=U_\alpha^{\lambda+\chi}(x)-U_\alpha^{\beta^\alpha_{D^c}(\lambda+\chi)}(x)$, $x\in\mathbb R^n$, inequality~(\ref{b1}) can be rewritten in the form
\begin{equation*}\Psi(x)\geqslant w_\lambda>0\quad(\xi-\lambda)\text{-a.e.~in\ }F.\end{equation*}
We can certainly assume that there is $y_0\in\partial D\cap S^{\xi-\lambda}_{\mathbb R^n}$, for if not, then the corollary is obvious. Since for every $\varepsilon>0$ it holds
$(\xi-\lambda)\bigl(B(y_0,\varepsilon)\bigr)>0$, one can choose $x_\varepsilon\in F\cap B(y_0,\varepsilon)$ so that
$\Psi(x_\varepsilon)\geqslant w_\lambda>0$. Therefore,
\begin{equation*}\limsup_{x\to y_0, \ x\in D}\,\Psi(x)\geqslant w_\lambda>0.\end{equation*}

On the other hand, $\Psi(x)\geqslant0$ for all $x\in\mathbb R^n$ and $\Psi(x)=0$ n.e.~in~$D^c$; hence,
 \[\liminf_{x\to y_0, \ x\in D^c}\,\Psi(x)=0.\]
Consequently, $\Psi$ is discontinuous on $\partial D\cap S^{\xi-\lambda}_{\mathbb R^n}$, and Lusin's type theorem for the $\alpha$-Riesz potentials (see~\cite[Theorem~3.6]{L})
establishes the corollary.\qed

\subsection{Proof of Corollary \ref{cor-0}}\label{proof2} Fix $\lambda_0\in\mathcal E^\xi_{g^\alpha_D,f}(F)$. We first assume that it solves
     Problem~\ref{pr2}, and let $w'_{\lambda_0}\in(0,\infty)$ be the number from (\ref{b1}) and~(\ref{b2}) for $f=0$ (see also~(\ref{strposss})). Then (\ref{b2}) can be
     rewritten in the form
\begin{equation*}\label{b2-0}U^{\lambda_0}_{\alpha}(x)\leqslant w'_{\lambda_0}+U^{\beta^\alpha_{D^c}\lambda_0}_{\alpha}(x)\quad\mbox{for all \ }x\in
S^{\lambda_0}_D.\end{equation*}
Note that the right-hand side of this relation is $\alpha$-superharmonic in~$\mathbb R^n$. Applying \cite[Theorems~1.27, 1.29]{L}, we see that, actually,
\begin{equation}\label{ppp}U^{\lambda_0}_{\alpha}(x)\leqslant w'_{\lambda_0}+U^{\beta^\alpha_{D^c}\lambda_0}_{\alpha}(x)\quad\mbox{for all \ }x\in\mathbb R^n,\end{equation}
which gives~(\ref{b2'}). Combining~(\ref{b2'}) with~(\ref{b1}) for $\lambda=\lambda_0$ and $w_\lambda=w_{\lambda_0}'$ results in~(\ref{b1'}).

Assuming now that both (\ref{b1'}) and (\ref{b2'}) hold for some $w'_{\lambda_0}\in(0,\infty)$, we conclude from Theorem~\ref{th-main''} that $\lambda_0$ solves
Problem~\ref{pr2}, as was to be proved.

Finally, let $\alpha<2$ and let $\lambda_0$ solve Problem~\ref{pr2}. To establish (\ref{identity}), assume on the contrary that there exists $x_0\in F$ such that $x_0\not\in
S^{\lambda_0}_D$. Then one can choose $r>0$ so that \[\overline{B}(x_0,r):=\{x\in\mathbb R^n: |x-x_0|\leqslant r\}\subset D\quad\text{and}\quad\overline{B}(x_0,r)\cap
S^{\lambda_0}_D=\varnothing.\] It follows that $(\xi-\lambda_0)\bigl(B(x_0,r)\cap F\bigr)>0$. Therefore, by~(\ref{b1'}),
\begin{equation}\label{bempty}U^{\lambda_0}_{\alpha}(x_1)=w'_{\lambda_0}+U^{\beta^\alpha_{D^c}\lambda_0}_{\alpha}(x_1)\quad\text{for some \ }x_1\in B(x_0,r)\cap
F.\end{equation}
  As $U^{\lambda_0}_{\alpha}(\cdot)$ is $\alpha$-harmonic in $B(x_0,r)$ and continuous on $\overline{B}(x_0,r)$, while
  $w'_{\lambda_0}+U^{\beta^\alpha_{D^c}\lambda_0}_{\alpha}(\cdot)$ is $\alpha$-super\-harmonic in~$\mathbb R^n$, we conclude from~(\ref{ppp}) and~(\ref{bempty}) with the help of
  \cite[Theorem~1.28]{L} that
\[U^{\lambda_0}_{\alpha}(x)=w'_{\lambda_0}+U^{\beta^\alpha_{D^c}\lambda_0}_{\alpha}(x)\quad m_n\mbox{-a.e.~in\ }\mathbb R^n.\]
This implies $w'_{\lambda_0}=0$, for $U^{\beta^\alpha_{D^c}\lambda_0}_{\alpha}(x)=U^{\lambda_0}_{\alpha}(x)$ holds n.e.~in~$D^c$, hence, also $m_n$-a.e.~in~$D^c$. A
contradiction.\qed

\subsection{Proof of Theorem~\ref{th-infty}} This theorem is a very particular case of \cite[Theorems~7.1, 7.2, 7.3]{ZPot2} (see also~\cite{Z5a} and Theorems~1,~2 and
Proposition~1 therein).\qed

\subsection{Proof of Corollary~\ref{cor-infty}}  Since (a) follows directly from
Theorem~\ref{infcap}, assume the conditions of assertion~(b) to hold. In the same manner as in Section~\ref{proof1}, then one can see that the number~$w_f$ from Theorem~\ref{th-infty} is strictly positive. Hence,  by~(\ref{F1}),
\[\Psi(x)\geqslant w_f>0\quad\mbox{n.e.~in\ }F,\]
where $\Psi$ has been defined in Section~\ref{proof1}. We can certainly assume that there is $y_0\in\partial D\cap{\mathrm C\ell}_{\mathbb R^n}\breve{F}$, for if not, then (b)
is obvious. For every $\varepsilon>0$, it holds $C_\alpha\bigl(B(y_0,\varepsilon)\cap\breve{F}\bigr)>0$, and therefore one can choose $x_\varepsilon\in
B(y_0,\varepsilon)\cap\breve{F}$ so that $\Psi(x_\varepsilon)\geqslant w_f>0$. This yields
 \begin{equation*}\limsup_{x\to y_0, \ x\in D}\,\Psi(x)\geqslant w_f>0.\end{equation*}
 Likewise as in Section~\ref{proof1}, we can thus see that $\Psi$ is discontinuous on $\partial D\cap{\mathrm C\ell}_{\mathbb R^n}\breve{F}$, and Lusin's type theorem for the
 $\alpha$-Riesz potentials establishes the corollary.\qed

\subsection{Proof of Corollary~\ref{cor-0'-infty}} Let $w:=w_f$ be the number from Theorem~\ref{th-infty} for $f=0$; then $w>0$.
Proof of the statement that $\lambda_F\in\mathcal E_{g^\alpha_D}^+(F,1)$ solves Problem~\ref{pr2} if and only if both~(\ref{w0-infty}) and~(\ref{w1-infty}) hold is
based on Theorem~\ref{th-infty} and runs in a way similar to that in the proof of Corollary~\ref{cor-0}. In particular, as an application of \cite[Theorems~1.27, 1.29]{L}, we
conclude from~(\ref{F2}) with $f=0$ that, if $\lambda_F$ solves Problem~\ref{pr2}, then
\begin{equation}\label{ppp-new}U^{\lambda_F}_{\alpha}(x)\leqslant w+U^{\beta^\alpha_{D^c}\lambda_F}_{\alpha}(x)\quad\mbox{for all \ }x\in\mathbb R^n.\end{equation}

The very last statement of the corollary is obtained from~(\ref{w0-infty}) with the help of standard arguments (see, e.g., \cite[pp.~137--138]{L}), based on the strict positive
definiteness of the $\alpha$-Green kernel.\qed

\subsection{Proof of Corollary~\ref{cor-sup}}
Having first assumed $\alpha<2$, we start by showing that
\begin{equation}\label{lesssup}U_{g^\alpha_D}^{\lambda_F}(x)<w\quad\text{for all \ }x\in D\setminus S_D^{\lambda_F}.\end{equation}
Suppose to the contrary that (\ref{lesssup}) is not satisfied for some $x_0\in D\setminus S^{\lambda_F}_D$.
Then $U_{g^\alpha_D}^{\lambda_F}(x_0)=w$ in accordance with~(\ref{w1-infty}), or equivalently
\begin{equation}\label{rrr}U_\alpha^{\lambda_F}(x_0)=w+U_\alpha^{\beta^\alpha_{D^c}\lambda_F}(x_0).
\end{equation}
Choose $\varepsilon>0$ so that $\overline{B}(x_0,\varepsilon)\subset D\setminus S^{\lambda_F}_D$. Since then
$U_\alpha^{\lambda_F}(\cdot)$ is $\alpha$-harmonic in~$B(x_0,\varepsilon)$ and continuous
on $\overline{B}(x_0,\varepsilon)$, while $w+U_\alpha^{\beta^\alpha_{D^c}\lambda_F}(\cdot)$
is $\alpha$-super\-harmonic in~$\mathbb R^n$, we conclude
from~(\ref{ppp-new}) and~(\ref{rrr}) with the help of~\cite[Theorem~1.28]{L} that
\begin{equation*}\label{dr}U_\alpha^{\lambda_F}(x)=w+U_\alpha^{\beta^\alpha_{D^c}\lambda_F}(x)\quad
m_n\text{-a.e.~in \ }\mathbb R^n.
\end{equation*}
As $U^{\beta^\alpha_{D^c}\lambda_F}_{\alpha}(x)=U^{\lambda_F}_{\alpha}(x)$ n.e.~in~$D^c$, we thus get $w=0$. A contradiction.

We next proceed by proving the former identity in (\ref{desc-sup}). Let, on the contrary,
there exist $x_1\in\breve{F}$ such that $x_1\not\in S^{\lambda_F}_D$,
and let $V\subset D\setminus S^{\lambda_F}_D$ be an open neighborhood
of~$x_1$. Then, by~(\ref{lesssup}), $U_{g^\alpha_D}^{\lambda_F}(x)<w$
for all $x\in V$. On the other hand, since $V\cap F$ has
nonzero capacity, $U_{g^\alpha_D}^{\lambda_F}(x_2)=w$ for some
$x_2\in V$ by~(\ref{w0-infty}). The contradiction obtained shows that, indeed, $S^{\lambda_F}_D=\breve{F}$. Substituting this identity into~(\ref{lesssup})
establishes~(\ref{less}) for $\alpha<2$.

In the rest of the proof, $\alpha=2$.
To verify~(\ref{less}), assume, on the contrary, that it does not hold for some~$x_3$ in the domain $D_0:=D\setminus\breve{F}$. According to~(\ref{w1-infty}), then
$U_{g^2_D}^{\lambda_F}(x_3)=w$, which in view of the harmonicity of
$U_{g^2_D}^{\lambda_F}$ in~$D_0$ implies, by the maximum
principle, that
\begin{equation*}\label{incr}U_{g^2_D}^{\lambda_F}(x)=w\quad\text{for all \
}x\in D_0.\end{equation*}
Thus,
\[\lim_{x\to z, \ x\in D_0}\,U_{g^2_D}^{\lambda_F}(x)=w>0\quad\text{for all \ } z\in\partial D_0.\]
Since $C_\alpha(\partial D\cap\partial D_0)>0$ in consequence of Corollary~\ref{cor-infty},~(b), Lusin's type theorem for the Newtonian potentials shows that the preceding
relation is impossible.

In view of~(\ref{w0-infty}), \cite[Theorem~1.13]{L} yields $\lambda_F\bigl|_{{\rm Int}\,F}=0$, and so
$S^{\lambda_F}_D\subset\partial_D\breve{F}$.
Thus, if we prove the converse inclusion, the latter identity in~(\ref{desc-sup}) follows. Assume, on the contrary, it
not to hold; then one can choose a point
$y\in\partial_D\breve{F}$ and a neighborhood $V_1\subset D$
of~$y$ so that $V_1\cap S^{\lambda_F}_D=\varnothing$. As $V_1\cap F$ has nonzero capacity, we see from~(\ref{w0-infty}) that there exists $y_1\in V_1$ such that
$U_{g^2_D}^{\lambda_F}(y_1)=w$. Taking~(\ref{w1-infty}) into account and applying the maximum principle to the harmonic in~$V_1$ function~$U_{g^2_D}^{\lambda_F}$, we thus have
$U_{g^2_D}^{\lambda_F}(x)=w$ for all $x\in V_1$.
This contradicts~(\ref{less}), because $V_1\cap D_0\ne\varnothing$.\qed

\subsection{Proof of Theorem~\ref{th-dual'}}
Since $U^\xi_{g^\alpha_D}(x)$ is continuous, so is $U^{\lambda_0}_{g^\alpha_D}(x)$. Indeed,
\[ U^{\lambda_0}_{g^\alpha_D}(x)=U^\xi_{g^\alpha_D}(x)-U^{\xi-\lambda_0}_{g^\alpha_D}(x),\]
which implies that $U^{\lambda_0}_{g^\alpha_D}(x)$ is both lower semicontinuous and upper semicontinuous. Next, since $\lambda_0$ solves the non-weighted Problem~\ref{pr2} with the constraint~$\xi$, both~(\ref{b1'}) and~(\ref{b2'}) are fulfilled. As $U^{\lambda_0}_{g^\alpha_D}(\cdot)$ is continuous, equality in~(\ref{b1'}) holds in fact everywhere on~$S^{\xi-\lambda_0}_D$. This allows us to rewrite \eqref{b1'} and \eqref{b2'} respectively as
\begin{align*} U^{\xi-\lambda_0}_{g^\alpha_D}(x)-U^{\xi}_{g^\alpha_D}(x)&=-w'_{\lambda_0} \quad \mbox{on \ }S^{\xi-\lambda_0}_D,\\
U^{\xi-\lambda_0}_{g^\alpha_D}(x)-U^{\xi}_{g^\alpha_D}(x)&\geqslant -w'_{\lambda_0} \quad\mbox{on\ }D,\end{align*}
or equivalently, in the notations accepted in Section~\ref{sec-var7},
\begin{align*} W^{\theta}_{g^\alpha_D,f}(x)&=-qw'_{\lambda_0} \quad \mbox{on\ }S^{\theta}_D,\\
W^{\theta}_{g^\alpha_D,f}(x)&\geqslant -qw'_{\lambda_0} \quad\mbox{on\ }D.\end{align*}
This establishes relations (\ref{D1}) and (\ref{D2}). Since $\theta\in\mathcal E^{q\xi}_{g^\alpha_D,f}(F,1)\subset\mathcal E^+_{g^\alpha_D,f}(F,1)$, application of  Theorems~\ref{th-main''} and~\ref{th-infty} therefore completes the proof.\qed

\section{Examples}\label{sec-ex}

In this section, $n=3$ and $x=(x_1,x_2,x_3)$ is a point in $\mathbb R^3$. In the following Examples\/~\ref{ex1}--\ref{ex6}, consider $0<\alpha\leqslant2$, $D:=B(0,1)$ and
$A_2:=D^c$; then $A_2$ is not $\alpha$-thin at~$\omega$.

\begin{exm}\label{ex1} Write $E:=\bigl\{x\in B(0,1): \ 0\leqslant x_1<1, \ x_2=x_3=0\bigr\}$.
Since $C_\alpha(E)=0$, Lemma~\ref{lemma:equiv} yields $C_{g^\alpha_D}(E)=0$. Consequently, there exists a neighborhood~$F$ of~$E$, closed in~$D$, with
$0<C_{g^\alpha_D}(F)<\infty$.
We can certainly assume that $\partial D\cap{\mathrm C\ell}_{\mathbb R^3}F=\{(1,0,0)\}$. Consider an external field~$f$ such that $f(x)<\infty$ n.e.~in~$F$ unless Case~II holds.
Application of Lemma~\ref{autom}, Theorems~\ref{lemma-main} and~\ref{th-main'} then shows that, in both Cases~I and~II, Problem~\ref{pr} is (uniquely)
solvable for every $\sigma\in\mathcal A(F)\cup\{\infty\}$.
\end{exm}

\begin{exm}\label{ex5}Let $F=D$. Define $\xi:=m_3|_F$; then $\xi\in\mathfrak C_0(F)$ and has finite $\alpha$-Riesz energy and thus it is admissible (see Definition~\ref{admissible}). Consider an external field~$f$ such that Case~I holds and $f(x)<\infty$ $m_3|_F$-a.e. Hence, by Lemma~\ref{autom}, assumptions~\eqref{riesz} and~\eqref{green} hold and so we can apply Theorems~\ref{silv-bound}
and~\ref{th-main'} to conclude that Problem~\ref{pr} is solvable; that is, no short-circuit between the conductors~$F$ and~$D^c$ occurs, though they touch each other over
the whole sphere~$S(0,1)$.\end{exm}

\begin{exm}\label{ex6}Let $F=S(x_0,1/2)\cap D$, where $x_0=(1/2,0,0)$. Consider an external field~$f$ such that Case~I holds and $f(x)<\infty$ $m_2|_F$-a.e. We further assume $1<\alpha\le 2$.  Define $\xi:=m_2|_F$; then (since $\alpha>1$) $\xi$ has finite $\alpha$-Riesz energy and so, as in the previous example, we can apply
Theorems~\ref{silv-bound} and~\ref{th-main'} to obtain the solvability of Problem~\ref{pr}.\end{exm}

\begin{exm}\label{ex4}
Let $\alpha=2$, $D=\bigl\{x\in\mathbb R^3: \  x_1>0\bigr\}$ and $F=\bigl\{x\in D: \ x_1=1\bigr\}$; then $A_2=D^c$ is not $2$-thin at~$\omega$, while
$C_{g^2_D}(F)=\infty$, for $C_2(F)=\infty$ by~\cite[Chapter~II, Section~3, n$^\circ$\,14]{L}. Let, in addition, Case~II with $\zeta\geqslant0$ hold. Then, by
Theorems~\ref{infcap} and~\ref{th-main'}, Problem~\ref{pr} is nonsolvable for every $\sigma\in\mathcal A(F)\cup\{\infty\}$ such that
$\sigma\geqslant\xi_0$, where $\xi_0\in\mathfrak C(F)\setminus\mathfrak C_0(F)$ is properly chosen. Thus, for these~$\sigma$, a short-circuit between~$F$ and~$D^c$
occurs at~$\omega$.
  To construct a constraint which would not allow such a short-circuit, consider $K_k$, $k\in\mathbb N$, where $K_k:=\bigl\{x\in F: \ (k-1)^2\leqslant x_2^2+x_3^2\leqslant
  k^2\bigr\}$ for all $k\geqslant 2$ and $K_1:=\bigl\{x\in F: \ x_2^2+x_3^2\leqslant 1\bigr\}$, and write
  \[\xi:=\sum_{k\in\mathbb N}\,\frac{m_2|_{K_k}}{k^3}.\]
Then $\xi$ is bounded and admissible, and so we can again use Theorems~\ref{silv-bound} and~\ref{th-main'} to see that Problem~\ref{pr} for this constraint~$\xi$ is solvable.
\end{exm}

\begin{figure}[htbp]
\begin{center}
\vspace{-.1in}

\includegraphics[width=4in]{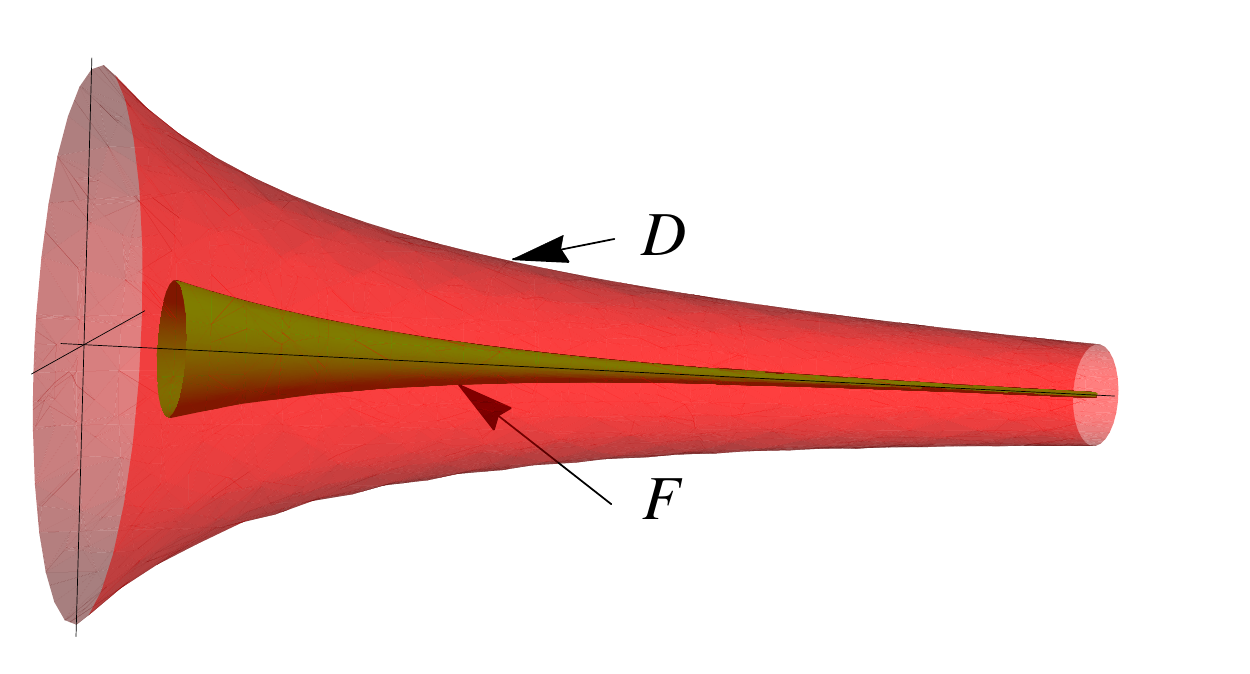}

\vspace{-.3in}
\caption{The condenser for Example~\ref{ex2}.}
\label{default}
\end{center}
\end{figure}

\begin{exm}\label{ex2} Let $\alpha=2$, Case~II with $\zeta\geqslant0$ hold, and let $F$ and~$D$ be defined by
\begin{align*}
F:=\bigl\{x&\in\mathbb R^3:&2\leqslant x_1&<\infty,&x_2^2+x_3^2&=\rho_1^2(x_1),&\text{where \ }\rho_1(x_1)&=\exp(-x_1)\bigr\},\\
D:=\bigl\{x&\in\mathbb R^3:&1<x_1&<\infty,&x_2^2+x_3^2&<\rho_2^2(x_1),&\text{where \ }\rho_2(x_1)&=x_1^{-1}\bigr\}.\end{align*}
Then $A_2=D^c$ is not $2$-thin at~$\omega$, while $C_{g^2_D}(F)=\infty$, for $C_2(F)=\infty$ by~\cite{Z0}.
      Hence, by Theorems~\ref{infcap} and~\ref{th-main'}, Problem~\ref{pr} is nonsolvable for every $\sigma\in\mathcal A(F)\cup\{\infty\}$ such that
      $\sigma\geqslant\xi_0$, where $\xi_0\in\mathfrak C(F)\setminus\mathfrak C_0(F)$ is properly chosen. However, Problem~\ref{pr} with $\xi:=m_2|_F$ is already
      solvable, which is seen from Theorems~\ref{silv-bound} and~\ref{th-main'}.
 \end{exm}

\end{document}